\documentclass[a4paper, 12pt]{article}

\usepackage{amsmath,amssymb,amsfonts}
\usepackage{amsthm}
\usepackage[abbrev]{amsrefs}
\usepackage[top=20truemm,bottom=20truemm,left=20truemm,right=20truemm]{geometry}
\usepackage{empheq}
\usepackage{color}
\usepackage[title]{appendix}

\numberwithin{equation}{section}
\theoremstyle{plain}
\newtheorem{theorem}{Theorem}[section]
\newtheorem{lemma}[theorem]{Lemma}
\newtheorem{corollary}[theorem]{Corollary}
\newtheorem{proposition}[theorem]{Proposition}
\theoremstyle{definition}

\newtheorem{remark}[theorem]{Remark}

\newcommand{\Complex}{ \mathbb{C} }

\newcommand{\Integer}{ \mathbb{Z} }

\newcommand{\Real}{ \mathbb{R} }

\newcommand{\Set}[2]{\left \{#1 \, ; \, #2 \right \}}

\newcommand{\Torus}{ \mathbb{T} }

\title{Data Assimilation to the Primitive Equations with $L^p$-$L^q$-based Maximal Regularity Approach}

\author{Ken Furukawa \thanks{Institute of physical and chemical research (RIKEN), ken.furukawa@riken.jp} 
\thanks{The author was partly supported by RIKEN Pioneering Project “Prediction for Science” and JSPS Grant-in-Aid for Young Scientists (No. 22K13948).}}

\begin{document}

\maketitle

\abstract{
    In this paper, we show mathematical justification of the data assimilation of nudging type in $L^p$-$L^q$ maximal regularity settings.
    We prove that the approximate solution of the primitive equations by data assimilation converges to the true solution with exponential order on the Besov space $B^{2/q}_{q,p}(\Omega)$ in the periodic layer domain $\Omega = \Torus^2 \times (-h, 0)$.
}


\section{Introduction} \label{intro}
The primitive equations are
\begin{equation} \label{eq_primitive}
    \begin{alignedat}{3}
        \partial_t v - \Delta v + u \cdot \nabla v + \nabla_H \pi
        & = f
        & \quad \text{in} \quad
        & \Omega \times (0, \infty), \\
        \partial_3 \pi
        & = 0
        & \quad \text{in} \quad
        & \Omega \times (0, \infty), \\
        \mathrm{div} \, u
        & =0
        & \quad \text{in} \quad
        & \Omega \times (0, \infty), \\
        v(0)
        & =v_0
        & \quad \text{in} \quad
        & \Omega,
    \end{alignedat}
\end{equation}
where $T>0$, $u$ is a unknown vector field with initial data $v_0$, $\pi$ are unknown scalar functions.
The domain $\Omega = \Torus^2 \times (-l, 0)$ is the periodic layer.
The differential operators $\nabla_H = (\partial_1, \partial_2)^T$, $\mathrm{div}_H = \nabla_H \cdot$, and $\Delta_H = \nabla_H \cdot \nabla_H$ are the horizontal gradient, the horizontal divergence, and the horizontal Laplacian, respectively.
The vertical velocity $w$ is given by
\begin{align*}
    w (x^\prime, x_3, t)
    = - \int_{-l}^0
        \mathrm{div}_H \, v (x^\prime, z, t)
    dz.
\end{align*}
We write $\Gamma_l$, $\Gamma_b$, and $\Gamma_u$ to denote the lateral, bottom, upper boundaries.
We impose the periodicity for $u, \pi$ on the lateral boundaries and
\begin{gather}\label{eq_bound_conditions}
    \begin{split}
        \partial_3 v = 0,
        \quad w = 0
        \quad \text{on}
        \quad \Gamma_u, \\
        v = 0,
        \quad w = 0
        \quad \text{on}
        \quad \Gamma_b.
    \end{split}
\end{gather}
The primitive equations are a fundamental model of geographic flows.
The global well-posedness of the primitive equations in $H^1$ was established by Cao and Titi \cite{CaoTiti2007}.
They proved this by combining the local well-posedness in $H^1$ with $H^1-$ a priori estimate.
The local well-posedness was proved by Guill\'{e}n-Gonz\'{a}lez, Masmoudi, and Rodr\'{i}guez-Bellido \cite{GuillenMasmoudiRodriguez2001}.
There are generalizations of Cao and Titi's results.
Hieber and Kashiwabara \cite{HieberKashiwabara2016} proved the global well-posedness in Lebesgue spaces $L^p$-settings based on the theory of analytic semigroup.
They used the equivalent equations to (\ref{eq_primitive}) such that
\begin{equation} \label{eq_primitive_evo}
    \begin{alignedat}{3}
        \partial_t v - P \Delta v + P (u \cdot \nabla v)
        & = P f
        & \quad \text{in} \quad
        & \Omega \times (0, \infty), \\
        \mathrm{div}_H \, \overline{v}
        & =0
        & \quad \text{in} \quad
        & \Omega \times (0, \infty), \\
        v(0)
        & =v_0
        & \quad \text{in} \quad
        & \Omega,
    \end{alignedat}
\end{equation}
where $P: L^q(\Omega)^2 \rightarrow L^q_{\overline{\sigma}}(\Omega)$ is the hydrostatic Helmholtz projection given by $P = I + \nabla_H (- \Delta_H)^{-1} \mathrm{div}_H$.
We write $C^\infty(\Omega)$ to denote the set of horizontally periodic $C^\infty$-functions for the horizontal variable $x^\prime$.
We write the Lebesgue space $L^q(\Omega)$, the Sobolev space $H^{s, q}(\Omega)$, and the Besov space $B^{s}_{q, p}(\Omega)$ to denote the completion of $C^\infty(\Omega)$ by the standard $L^q$-, $H^{s, q}$-, and $B^s_{q, p}$- norms for $s \geq 0$ and $1 < q < \infty$.
We denote the $\overline{\mathrm{div}_H}$-free $L^q$-vector fields $L^q_{\overline{\sigma}}(\Omega)$ by
\begin{align*}
    L^q_{\overline{\sigma}}(\Omega) 
    = \overline{
        \Set{
            \varphi \in C^\infty(\Omega)^2
        }{
            \mathrm{div}_H \overline{\varphi} = 0
        }
    }^{\Vert \cdot \Vert_{L^q(\Omega)^2}},
\end{align*}
where $\overline{\varphi} = \int_{-l}^0 \varphi(\cdot, \cdot, z) dz/l$ is the horizontal average.
We analogously denote $\overline{\mathrm{div}_H}$-free Sobolev spaces $H^{s, q}_{\overline{\sigma}}(\Omega)$ and Besov space $B^{s}_{q, p, \overline{\sigma}}(\Omega)$.
Note that $P: L^q(\Omega)^2 \rightarrow L^q_{\overline{\sigma}}(\Omega)$ is bounded.
Giga, Gries, Hieber, Hussein, and Kashiwabara \cite{GigaGriesHieberHusseinKashiwabara2017_analiticity} proved the global well-posedness in $L^p$-$L^q$ settings under various boundaries conditions, $i. e.$ the periodic, Neumann, Dirichlet, Dirichlet-Neumann mixed boundary conditions.

For the basic flow $v$, which is the solution to the primitive equations, we consider the data assimilation (DA) problem
\begin{equation} \label{eq_nudging}
    \begin{alignedat}{3}
        \partial_t \tilde{v} - P \Delta \tilde{v} + P \left(
            \tilde{u} \cdot \nabla \tilde{v}
        \right)
        & = P J_\delta f + \mu P (J_\delta v - J_\delta \tilde{v})
        & \quad \text{in} \quad
        & \Omega \times (0, \infty),\\
        \mathrm{div}_H \, \overline{v}
        & = 0,
        & \quad \text{in} \quad
        & \Omega \times (0, \infty),\\
        \tilde{v}(0)
        & = \tilde{v}_0
        & \quad \text{in} \quad
        & \Omega,
    \end{alignedat}
\end{equation}
where $\mu >0$ is a constant called an inflation parameter, $\delta >0$ is a constant representing the inverse of the observation density, and the bounded linear operator $J_\delta$ is a generalization of the low-path operator satisfying
\begin{gather} \label{eq_J}
    \begin{split}
        \Vert
            J_\delta f
        \Vert_{L^q (\Omega)^2}
        & \leq C
        \Vert
            f
        \Vert_{L^q (\Omega)^2}, \quad \text{for} \quad f \in L^q (\Omega)^2, \\
        \Vert
            J_\delta f - f
        \Vert_{L^q (\Omega)^2}
        & \leq C \delta
        \Vert
            \nabla f
        \Vert_{L^q (\Omega)^2} \quad \text{for} \quad f \in H^{1, q} (\Omega)^2,
    \end{split}
\end{gather}
for some constant $C>0$.
Note $J_\delta$ is independent of time variable.
This kind of DA procedures are called the nudging.
A typical example of $J_\delta$ is a cube-wise averaging operator, where each cube is a peace of homogeneous small cubic decomposition of $\Omega$ with the radius $O(\delta^{1/3})$. 
The term $J_\delta v - J_\delta \tilde{v}$ in (\ref{eq_nudging}) behaves like a forcing term to make $\tilde{v}$ converge to $v$.
A basic aim of DA is to predict the true state $v$ by using observation $J_\delta v$.
We remark that the information of $v$ itself is not used directly because we never obtain the perfect observation in the real world.
This can be seen from the right-hand side of (\ref{eq_nudging}).
In the right-hand side there are information of the observations and no direct information about $f$ and $v$.

The DA is typically used in meteorology to forecast many physical variables in the atmosphere and the ocean, such as velocity field, temperature field, and other physically meaningful quantities.
DA has strongly related to the partial differential equations of geographic flows.
Azouani, Olson, and Titi \cite{AzouaniOlsonTiti2014} gave a mathematical framework to consider the DA, and they showed the solution to the DA equations of the two-dimensional Navier-Stokes equation converges to the true solution in the energy space.
Albanez, Nussenzveig, and Titi \cite{AlbanezNussenzveigTiti2015} showed the same type of convergence result for the Navier-Stokes $\alpha$-model.
Pei \cite{Pei2019} showed the convergence of the DA equations to the primitive equations in the energy space.
There were no convergence results in $L^q$-frameworks.
The aim of this paper is to show that the solution to (\ref{eq_nudging}) converges to (\ref{eq_primitive}) with exponential order in $L^p$-$L^q$-based maximal regularity settings.

We set $V = v - \tilde{v}$, $W = w - \tilde{w}$, $U = (V, W)$, $V_0 = v_0 - \tilde{v}_0$, and $F = f - J_\delta f$.
To show the convergence of $\tilde{v}$ to the true solution $v$, we consider the equations of the difference
\begin{equation} \label{eq_diff}
    \begin{aligned}
        \partial_t V - P \Delta V + P \left(
            u \cdot \nabla V + U \cdot \nabla v - U \cdot \nabla V
        \right)
        & = P F - \mu P J_\delta V
        & \text{in} \quad
        & \Omega \times (0, \infty), \\
        \mathrm{div}_H \, \overline{V}
        & = 0
        & \text{in} \quad
        & \Omega \times (0, \infty),\\
        V(0)
        & = V_0
        & \text{in} \quad
        & \Omega.
    \end{aligned}
\end{equation}
For a Banach space $X$, we denote by $L^p_\eta (0, T; X)$ and $H^{1, p}_\eta (0, T; X)$ the $X$-valued $\eta$-weighted Lebesgue space and Sobolev space, respectively, such that
\begin{gather*}
    L^p_\eta (0, T; X)
    = \left \{
        f \in L^1_{loc} (0, T; X)
        \, : \,
        \Vert
            u
        \Vert_{L^p_\eta (0, T; X)}
        < \infty
    \right \}, \\
    H^{m,p}_\eta (0, T; X)
    = \left \{
        f, \, \partial_t f, \cdots, \partial_t^m f \in L^1_{loc} (0, T; X)
        \, : \,
        f, \, \partial_t f \cdots, \partial_t^m f \in L^p_\eta (0, T; X)
    \right \}, \\
    \Vert
        u
    \Vert_{L^p_\eta (0, T; X)}
    := \left(
        \int_0^T
            \left(
                t^{1 - \eta} \Vert
                    f(t)
                \Vert_{X}
            \right)^p
        \frac{dt}{t}
    \right)^{\frac{1}{p}}.
\end{gather*}
We use the standard modification when $p, q = \infty$.
We write $A_q = P (- \Delta_q)$ to denote the hydrostatic Stokes operator with the domain
\begin{align}
    D(A_q) = \Set{
        \varphi \in H^{2, q}(\Omega)^2 \cap L^q_{\overline{\sigma}}(\Omega)
    }{
        \text{$\varphi$ satisfies (\ref{eq_bound_conditions})}
    }.
\end{align}
The space of initial trace $X_{\theta, p, q} = (L^q_{\overline{\sigma}}(\Omega), D(A_q))_{\theta, p}$ is characterized by
\begin{align}
    X_{\theta, p, q}
    = \left \{
        \begin{array}{lll}
            \Set{
                v \in B^{2\theta}_{q, p}(\Omega)^2 \cap L^q_{\overline{\sigma}}(\Omega)
            }{
                \partial_3 v |_{x_3 = 0} = 0, \, v|_{x_3 = -l} = 0
            }
            & \text{if}
            & \frac{1}{2} + \frac{1}{2q} < \theta < 1,\\
            \Set{
                v \in B^{2\theta}_{q, p}(\Omega)^2 \cap L^q_{\overline{\sigma}}(\Omega)
            }{
                v|_{x_3 = -l} = 0
            }
            & \text{if}
            &  \frac{1}{2q} < \theta < \frac{1}{2} + \frac{1}{2q},\\
            B^{2\theta}_{q, p}(\Omega)^2 \cap L^q_{\overline{\sigma}}(\Omega)
            & \text{if}
            & \theta < \frac{1}{2q},
        \end{array}
    \right .
\end{align}
the reader refers to \cite{GigaGriesHieberHusseinKashiwabara2017_analiticity}.
The main result of this paper is
\begin{theorem} \label{thm_main_thoerem}
    Let $1 < p, q < \infty$ satisfy $\frac{1}{p} + \frac{1}{q} \leq 1$, and $\eta = \frac{1}{p} + \frac{1}{q}$.
    Let $v_0, V_0 \in X_{1/q, p, q}$.
    Let $f \in L^p_\eta (0, \infty; L^q(\Omega)^2)$ satisfy
    \begin{align*}
        \Vert
            e^{\gamma_0 t} f
        \Vert_{L^p_\eta (0, \infty; L^q(\Omega)^2)}
        + \Vert
            \partial_t f
        \Vert_{L^2_{loc}(0, \infty; L^2(\Omega)^2)}
        + \Vert
            e^{\gamma_0 t} f
        \Vert_{L^2(0, \infty; H^1(\Omega)^2)}
        < \infty,
    \end{align*}
    for some constant $\gamma_0 > 0$.
    Assume $v \in C(0, \infty; X_{1/q, p, q}) \hookrightarrow C(0, \infty; B^{2/q}_{q, p}(\Omega)^2)$ is the solution to (\ref{eq_primitive}) obtained by \cite{GigaGriesHieberHusseinKashiwabara2017_analiticity} satisfying
    \begin{gather} \label{eq_bounds_for_v}
        \begin{split}
            \Vert
                e^{\gamma_1 t} \partial_t v
            \Vert_{L^p_\eta (0, \infty; L^q(\Omega)^2)}
            + \Vert
                e^{\gamma_1 t} \nabla^2 v
            \Vert_{L^p_\eta (0, \infty; L^q(\Omega)^2)}
            < \infty, \\
            \Vert
                v (t)
            \Vert_{X_{1/q, p, q}}
            \leq C_0 e^{- \gamma_1 t},
        \end{split}
    \end{gather}
    for some constant $C_0>0$ and $\gamma_1 < \gamma_0$.
    Then there exit $\mu_0, \delta_0>0$, if $\mu \geq \mu_0$ and $\delta \leq \delta_0$, there exits a unique solution $V \in C(0, \infty;X_{1/q, p, q})$ to (\ref{eq_nudging}) such that
    \begin{align*}
        \Vert
            e^{\mu_\ast t} \partial_t V
        \Vert_{L^p_\eta(0, \infty; L^q(\Omega)^2)}
        + \Vert
            e^{\mu_\ast t} V
        \Vert_{L^p_\eta(0, \infty; H^{2,q}(\Omega)^2)}
        \leq C,
    \end{align*}
    for some constants $\gamma_0 < \mu_\ast < \gamma_1$ and $C>0$.
    Moreover, $V$ is exponentially stable in the sense
    \begin{align*}
        \Vert
            V (t)
        \Vert_{X_{1/q, p, q}}
        \leq C e^{ - \mu_\ast t}.
    \end{align*}
\end{theorem}
\begin{remark}
    \begin{enumerate}
        \item The assumption $f \in H^1(0, T; L^2(\Omega)^2)$ is used to obtain time regularity, which ensure the solution $V(t)$ belongs $H^2(\Omega)^2$ for a. a. $t \in (0, T)$.
        \item The existence of the solution $v$ to (\ref{eq_primitive_evo}) satisfying (\ref{eq_bounds_for_v}) have been proved by Giga $et \, al$ \cite{GigaGriesHieberHusseinKashiwabara2017_analiticity}.
        One can also see that the existence of the solution satisfying (\ref{eq_bounds_for_v}) from the proof of Theorem \ref{thm_main_thoerem} by taking $v \equiv 0$ and $\mu, \delta = 0$.
        \item The theorem is an extension of Pei's $L^2$-convergence result \cite{Pei2019}.
        When $p = q = 2$, Theorem \ref{thm_main_thoerem} implies $\tilde{v} \rightarrow v$ in $H^1(\Omega)^2$ as $t \rightarrow \infty$, which is stronger convergence than Pei's result.
    \end{enumerate}
\end{remark}

There are several steps to prove Theorem \ref{thm_main_thoerem}.
We put $K_\delta = I - J_\delta$.
We first show that the perturbed hydrostatic operator $\tilde{A}_{\mu, q} = A_q + \mu P J_\delta = (A_q + \mu I) - \mu P K_\delta$ admits the bounded $H^\infty$-calculus.
We know that $K_\delta: H^{1, q}(\Omega)^2 \rightarrow L^q(\Omega)^2$ is bounded by (\ref{eq_J}) and the hydrostatic Stokes operator $A_q = P (- \Delta)$ with $D(A_q) = H^{2, q}(\Omega)^2 \cap L^q_{\overline{\sigma}}(\Omega)$ is invertible and admits bounded $H^\infty$-calculus by \cite{GigaGriesHieberHusseinKashiwabara2017_analiticity}.
Then we see that the perturbed hydrostatic operator $\tilde{A}_{\mu, q}$ admits the bounded $H^\infty$-calculus if we take $\delta > 0$ so small that $\delta = O(\mu^{-1})$ by perturbation arguments, see the book by Pr\"{u}ss and Simonett \cite{PrussSimonett2016} for the definition and properties of bounded $H^\infty$-calculus.
Therefore, we find that $\tilde{A}_{\mu, q}$ with $D(\tilde{A}_{\mu, q}) = D(A_q)$ has the $L^p$-$L^q$ maximal regularity and generates an analytic semigroup $e^{- t \tilde{A}_{\mu, q}}$.
Note that, if we take $\mu$ sufficiently large, $\tilde{A}_{\mu, q}$ is exponentially stable, and the decay rate is larger than the minimum spectrum of $A_q$.

We construct the local-in-time strong solution to (\ref{eq_diff}) based on Banach's fixed point theorem in $L^p$-$L^q$ maximal regularity settings.
We follow the methods by Giga $et \, al.$ \cite{GigaGriesHieberHusseinKashiwabara2017_analiticity}.
We also prove $V \in H^{2, p}_{\eta, loc}(0, T; L^q(\Omega)^2) \cap H^{1, p}_{\eta, loc}(0, T; D(\tilde{A}_{\mu, q})) \hookrightarrow C(0, T; D(\tilde{A}_{\mu, q}))$ by using the assumption $f \in H^{1,p}_\eta (0, T; L^q(\Omega)^2)$, which is used to connect the $L^p$-$L^q$-based solutions to $H^1$-global solution.

We establish $H^2$- a priori estimate for $V$.
The strategy is based on Giga $et \, al$ \cite{GigaGriesHieberHusseinKashiwabara2017_analiticity}.
They showed $v \in C(0, T; D(A_2))$ by proving $H^1(0, T; D(A_2))$ bound using integration by parts and time regularity of $f \in H^1(0, T; L^q(\Omega)^2)$.
In contrast to this, Hieber and Kashiwabara \cite{HieberKashiwabara2016}  obtained $v \in C(0, T; D(A_2))$ using the spacial regularity assumption for $f \in L^2(0, T; H^1(\Omega)^2)$.
In our case, Hieber and Kashiwabara's scenario does not work because the term $J_\mu V$ in (\ref{eq_diff}) does not always belong to $L^2(0, T; H^1(\Omega)^2)$ even if $V$ is smooth for spacial variables.
Due to this reason, we assume time regularity for the external force.
The assumption $f \in L^2(0, T; H^1(\Omega)^2)$ is used to bound the convection term $U \cdot \nabla v$ in (\ref{eq_diff}).
We need not to estimate $U \cdot \nabla V$ and $u \cdot \nabla V$ to establish $H^1$- a priori estimate because these terms vanish by integration by parts.
However, the term $U \cdot \nabla v$ does not vanish by integration by parts, then we have to estimate this.
We need spacial regularity of the basic flow $v \in L^2(0, T; H^3(\Omega)^2)$ to estimate the convection term.
This additional regularity can be obtained if $f \in L^2(0, T; H^1(\Omega)^2)$, see Appendix A.
Combining with the local well-posedness and $H^1$- a priori estimate, we establish the global well-posedness in the maximal regularity space.
For the case $T = \infty$, we first derive $H^2$ decay of $V$ when $p = q = 2$.
For the other case we use the embedding $H^2(\Omega)^2 \hookrightarrow B^{2/q}_{q, p}(\Omega)^2$ to get smallness of $V(T_0)$ for some $T_0 > 0$.
By the choice of the critical time weight $\eta$, we can extend $V(t)$ to infinity.
Although, in the previous studies \cites{AzouaniOlsonTiti2014,Pei2019}, the authors construct the solution to the nudging equations directly, we show the existence and the uniqueness of the solution to (\ref{eq_nudging}) by contradiction argument, global well-posedness of the primitive equations, and Theorem \ref{thm_main_thoerem}.
Actually, we can show that there exists a unique solution $\tilde{v}$ to (\ref{eq_nudging}) on $H^{1,q}(t_0 + \tau, t_0; L^q(\Omega)) \cap L^q(t_0 + \tau, t_0; H^{2, q}(\Omega))$ for all $t_0 >0$ and small $\tau > 0$ with initial data $\tilde{v} (t_0)$.
On the other hand, vector field $v - V$ solves (\ref{eq_nudging}) in $H_\eta^{1,q}(0, \infty; L^q(\Omega)) \cap L^q_\eta(0, \infty; H^{2, q}(\Omega))$.
By the uniqueness, the vector field $v - V$ must be the global solution to (\ref{eq_nudging}).

In this paper we use the following notation and notion.
We write $L^q_x L^r_{x^\prime} (\Omega \times \Omega^\prime) = L^q(\Omega; L^r(\Omega^\prime))$ for domains $\Omega, \Omega^\prime$ and $x \in \Omega, x^\prime \in \Omega^\prime$.
We write
\begin{align*}
    \dot{H}^{m, q}(\Omega)
    = \Set{
        \varphi \in L^1_{loc}(\Omega)
    }{
        \Vert \varphi \Vert_{\dot{H}^{m}}
        := \sum_{|\alpha|=m} \Vert \partial^\alpha_x \varphi \Vert_{L^q(\Omega)} < \infty
    }.
\end{align*}
for $m \in \Integer_{\geq 0}$ and $q \in (1, \infty)$.
For a bounded operator $B: X \rightarrow Y$ for a Banach spaces $X, Y$, we write $\Vert B \Vert_{X \rightarrow Y}$ to denote the operator norm of $B$.
We denote the sector $\Sigma_\theta = \Set{\lambda \in \Complex}{ |\arg \lambda| < \theta}$ and define the set of analytic function on $\Sigma_\theta$ decaying at the origin and infinity such that
\begin{align*}
    H^\infty_0(\Sigma_\theta)
    = \Set{
        \varphi\,:\, \text{analytic on $\Sigma_\theta$}
    }{
        |\varphi(\lambda)| \leq C \left| \frac{\lambda}{1 + \lambda^2} \right|^\varepsilon \, \text{for $\lambda \in \Sigma_{\theta}$ and some $\varepsilon>0$}
    }.
\end{align*}
The space $H^\infty_0(\Sigma_\theta)$ is equipped with the uniform norm on $\Sigma_\theta$.
We say a sectorial operator $S$ with the domain $D(S) \subset X$ and the spectral angle $\theta_S \in (0, \pi)$ admits a bounded $H^\infty$-calculus on $X$ if
\begin{align*}
    \left \Vert
        \frac{1}{2 \pi i} \int_{\partial \Sigma_\theta}
            \varphi(\lambda) (\lambda - S)^{-1}
        d\lambda
    \right \Vert_{X \rightarrow X}
    \leq C \Vert
        \varphi
    \Vert_{H^\infty_0(\Sigma_\theta)}
\end{align*}
for all $\varphi \in H^\infty_0(\Sigma_\theta)$ with $\theta < \theta_S$.
In particular, the bounded $H^\infty$-calculus of $S$ on $L^q(\Omega)$ with the $H^\infty$-angle $\theta_{S, H^\infty} < \pi/2$ implies the maximal regularity estimate
\begin{align*}
    \Vert
        \psi
    \Vert_{H^{1, p}(0, T; L^q(\Omega))}
    + \Vert
        \psi
    \Vert_{L^{p}(0, T; D(S))}
    \leq C 
        \left(
        \Vert
            u_0
        \Vert_{(L^q(\Omega), D(S))_{1-/p, p}}
        + \Vert
            g
        \Vert_{L^p(0, T; D(S))}
        \right),
\end{align*}
for the the solution to
\begin{align*}
    \partial_t u + S u
    & = g \in L^p(0, T; L^q(\Omega)),\\
    u(0)
    & = u_0 \in (L^q(\Omega), D(S))_{1-/p, p}.
\end{align*}
The maximal regularity holds on the time-weight Lebesgue space $L^q_\eta(0, T; L^q(\Omega))$.
We denote the trace space $D_S(\theta, p)$ for a sectorial operator $S$ admitting a bounded $H^\infty$-calculus with $H^\infty$ angle $\phi_{H^\infty} \in [0, \pi/2)$ by
\begin{align*}
    D_S(\theta, p)
    & = \Set{
        x \in L^q_{\overline{\sigma}}(\Omega)
    }{
        \Vert
            x
        \Vert_{D_S(\theta, p)}
        < \infty
    }, \\
    \Vert
        x
    \Vert_{D_S(\theta, p)}
    & = \Vert
        x
    \Vert_{L^q(\Omega)^2}
    + \left( \int_0^\infty
            \Vert
                t^{1 - \theta} S e^{- t S} x
            \Vert_{L^q(\Omega)^2}^p
        \frac{dt}{t}
    \right)^{\frac{1}{p}}.
\end{align*}
We denote the Lebesgue-Sobolev mixed spaces $\mathbb{E}_{0, \eta, p, q, T}$ and $\mathbb{E}_{1, \eta, p, q, T}$ by
\begin{align*}
    \mathbb{E}_{0, \eta, p, q, T}
    & = L^{p}_\eta(0, T; L^q(\Omega)^2), \\
    \mathbb{E}_{1, \eta, p, q, T}
    & = H^{1, p}_\eta(0, T; L^q(\Omega)^2) \cap L^{p}_\eta(0, T; D(A_q)).
\end{align*}
\section{Linear Theory and Well-posedness}
\subsection{Bounded $H^\infty$-calculus of the linearized operator}
The linearized equation of (\ref{eq_diff}) is
\begin{equation} \label{eq_perturbed_hydrostatic_stokes}
    \begin{aligned}
        \partial_t v - P \Delta v + \mu P J_\delta v
        & = f
        & \quad \text{in} \quad \Omega \times (0, \infty) \\
        \mathrm{div}_H \, \overline{v}
        & = 0
        & \quad \text{in} \quad \Omega \times (0, \infty)
    \end{aligned}
\end{equation}
with initial data $v (0) = v_0$.
We show the maximal regularity estimates for (\ref{eq_perturbed_hydrostatic_stokes}).
\begin{lemma} \label{lem_H_infty_for_tilde_A}
    Let $q \in (1, \infty)$, $\mu > 0$, and $\delta > 0$.
    Then there exists a constant $\alpha > 0$, if $\mu \delta < \alpha$, the perturbed hydrostatic operator operator $\tilde{A}_{\mu, q} := A_q + \mu J_\delta$ with $D(\tilde{A}_{\mu, q}) = D(A)$ admits a bounded $H^\infty$-calculus with $H^\infty$ angle $\theta_{H^\infty} \in (0, \pi/2)$.
    Furthermore, $\tilde{A}_{\mu, q}$ generates an analytic semigroup $e^{- t \tilde{A}_{\mu, q}}$ with faster decay rate than $e^{- t A}$ i.e.
    \begin{align*}
        \Vert
            e^{- t \tilde{A}_{\mu, q}}
        \Vert_{L^q (\Omega)^2 \rightarrow L^q (\Omega)^2}
        \leq C e^{ - \mu_\ast t}
    \end{align*}
    for some $C, \mu_\ast > 0$ and $\mu_\ast$ can be taken larger than the minimum spectrum of $A$.
\end{lemma}.
\begin{corollary}
    Let $\frac{1}{p} < \eta \leq 1$
    Under the same assumption of Theorem \ref{lem_H_infty_for_tilde_A},
    \begin{align}
        \Vert
            e^{\mu_\ast t} \partial_t v
        \Vert_{\mathbb{E}_{0, \eta, p, q, T}}
        + \Vert
            e^{\mu_\ast t} \tilde{A}_{\mu, q} v
        \Vert_{\mathbb{E}_{0, p, q, T}}
        \leq C \left(
            \Vert
                v_0
            \Vert_{X_{\eta - 1/p, p, q}}
            + \Vert
                e^{\mu_\ast t} f
            \Vert_{\mathbb{E}_{0, p, q, T}}
        \right),
    \end{align}
    for the solution $v$ to $\partial_t v + \tilde{A}_{\mu, q} v = f$ with initial data $v_0 \in X_{\eta - 1/p, p, q}$ and $f$ satisfying $e^{\mu_\ast t} f \in \mathbb{E}_{0, \eta, p, q, T}$.
    The constant $C>0$ is independent of $T$.
\end{corollary}.

\begin{proof}[Proof of Lemma \ref{lem_H_infty_for_tilde_A}]
    We first invoke that the hydrostatic Stokes operator $A_q$ generates an analytic semigroup $e^{- t A_q}$ in $L^q_{\overline{\sigma}}(\Omega)$ and admits the bounded inverse, see \cite{HieberHusseingKashiwabara2016}.
    There exits a solution $\psi$ to $ \lambda \psi - A_q \psi = Pf$ such that
    \begin{align} \label{eq_resolvent_ineq_hydrostatic_stokes}
        |\lambda| \Vert
            \psi
        \Vert_{L^q (\Omega)^2}
        + \Vert
            \psi
        \Vert_{\dot{H}^2 (\Omega)^{2}}
        \leq C \Vert
            Pf
        \Vert_{L^q (\Omega)^2}
        \leq C \Vert
            f
        \Vert_{L^q (\Omega)^2},
    \end{align}
    for all $\lambda \in \overline{\Sigma_\theta}^c \cup \{0\}$ and all $\theta \in (0, \pi/2)$.
    Since $J_\delta = I + K_\delta$, we have the resolvent problem for $\tilde{A}_{\mu, q}$
    \begin{align} \label{eq_resolvent_problem_A+muJ}
        (\lambda + \mu) v + A_qv + \mu P K_\delta v = Pf,
    \end{align}
    By the assumption (\ref{eq_J}), we see that $K_\delta$ is a relatively compact perturbation to $A_q$ with domain $D(K_\delta) = H^{1, q} (\Omega)^2$ satisfying
    \begin{align*}
        \Vert
            \mu P K_\delta \varphi
        \Vert_{L^q(\Omega)^2}
        & \leq C \mu \delta \Vert
            \nabla \varphi
        \Vert_{L^q(\Omega)^2}
        \leq C \mu \delta \Vert
            \Delta \varphi
        \Vert_{L^q(\Omega)^2}^{\frac{1}{2}}
        \Vert
            \varphi
        \Vert_{L^q(\Omega)^2}^{\frac{1}{2}} \\
        & \leq C \mu \delta \Vert
            A_q\varphi
        \Vert_{L^q(\Omega)^2}^{\frac{1}{2}}
        \Vert
            \varphi
        \Vert_{L^q(\Omega)^2}^{\frac{1}{2}} \\
        & \leq \varepsilon \Vert
            A_q\varphi
        \Vert_{L^q(\Omega)^2}
        + \frac{1}{4 \varepsilon} \Vert
            \varphi
        \Vert_{L^q(\Omega)^2}
    \end{align*}
    for all $\varphi \in D (A_q)$ and some small $\delta,  \varepsilon \in (0, 1)$.
    Therefore $A_q+ \mu P K_\delta$ with the domain $D(A_q+ \mu P K_\delta) = D(A_q)$ generates an analytic semigroup $e^{- t(A_q+ \mu P K_\delta)}$ on $L^q_{\overline{\sigma}} (\Omega)$.
    We see that $\lambda = 0$ belongs to the resolvent of $A_q$, the embedding $D(A_q) \hookrightarrow D(B)$, and the inequality $\Vert \mu P K_\delta v \Vert_{L^q(\Omega)^2} \leq \Vert (-A_q)^{\frac{1}{2}+\varepsilon^\prime} v \Vert_{L^q(\Omega)^2}$ for all $v \in D(A_q)$, small $\delta >0$, and $ \varepsilon^\prime \in [0, 1/2)$.
    Using perturbation theory of $H^\infty$-calculus, $e.g.$ Lemma 4.1 in the book \cite{KunstmannWeis2004}, we see there exits a constant $\lambda_0 > 0$ such that $\lambda_0 + A_q+ \mu P K_\delta$ admits a bounded $H^\infty$-calculus on $L^q_{\overline{\sigma}} (\Omega)$ with the $H^\infty$-angle angle $\theta^\prime \in (0, \pi/2)$.
    In view of (\ref{eq_resolvent_problem_A+muJ}), if we take $\mu$ sufficiently large, we find $\tilde{A}_{\mu, q}$ generates an analytic semigroup $e^{- t \tilde{A}_{\mu, q}}$ with faster exponential decay than $e^{- t A_q}$ and also admits an bounded $H^\infty$-calculus in $L^q_{\overline{\sigma}}(\Omega)$.
\end{proof}
Hereafter, we fix the constant $\mu_\ast$ in Lemma \ref{lem_H_infty_for_tilde_A} and always assume $\mu > \mu_0$ and $\delta < \delta_0$ for large $\mu_0>0$ and small $\delta_0 >0$ so that the assertions of Lemma \ref{lem_H_infty_for_tilde_A} holds.

By the same kind of perturbation arguments as the proof of Lemma \ref{lem_H_infty_for_tilde_A}, we can show
\begin{proposition} \label{prop_H_infty_for_A_plus_J}
    Under the same assumption for $q, \mu, \delta$ as Lemma \ref{lem_H_infty_for_tilde_A}.
    There exists a constant $\alpha > 0$, if $\mu \delta < \alpha$, the perturbed hydrostatic operator operator $\tilde{\tilde{A}}_{\mu, q} = A_q + \mu J_\delta$ with $D(\tilde{\tilde{A}}_{\mu, q}) = D(A)$ admits the bounded $H^\infty$-calculus.
    Furthermore, $\tilde{\tilde{A}}_{\mu, q}$ generates an analytic semigroup $e^{- t \tilde{\tilde{A}}_{\mu, q}}$ such that
    \begin{align*}
        \Vert
            e^{- t \tilde{\tilde{A}}_{\mu, q}}
        \Vert_{L^q (\Omega)^2 \rightarrow L^q (\Omega)^2}
        \leq C e^{ c t}
    \end{align*}
    for some $C, c > 0$.
\end{proposition}
We define the exponentially-weighted maximal regularity space $\mathbb{E}_{1, \eta, p, q, T, e^{\gamma t}}$ equipped with the norm
\begin{align*}
    \Vert
        v
    \Vert_{\mathbb{E}_{1, \eta, p, q, T, e^{\gamma t}}}
    := \Vert
        e^{\gamma t} \partial_t v
    \Vert_{L^p_\eta(0, T ; L^q(\Omega)^2)}
    + \Vert
        e^{\gamma t} v
    \Vert_{L^p_\eta(0, T ; D(A_{\mu, q}))}.
\end{align*}
We find that
\begin{align*}
    \Vert
        \partial_t \left(e^{\gamma t} v\right)
    \Vert_{L^p_\eta(0, T ; L^q(\Omega)^2)}
    & \leq \gamma \Vert
        e^{\gamma t} v
    \Vert_{L^p_\eta(0, T ; L^q(\Omega)^2)}
    + \Vert
        e^{\gamma t} \partial_t v
    \Vert_{L^p_\eta(0, T ; L^q(\Omega)^2)} \\
    & \leq C \Vert
        v
    \Vert_{\mathbb{E}_{1, \eta, p, q, T, e^{\gamma t}}},
\end{align*}
where the constant $C > 0$ is independent of $T$.

\begin{proposition} \label{prop_pointwise_estimate_semigroup}
    Let $q \in (1, \infty)$ and $\theta \in [0, 1]$.
    Then 
    \begin{align}
        \Vert
            \tilde{A}_{\mu, q}^\theta e^{- t \tilde{A}_{\mu, q}} f
        \Vert_{L^q(\Omega)^2}
        \leq C t^{ - \theta} e^{- \mu_\ast t} \Vert
            f
        \Vert_{L^q(\Omega)^2}
    \end{align}
    for all $f \in L^q_{\overline{\sigma}}(\Omega)^2$.
\end{proposition}
\begin{proof}
    Since $\tilde{A}_{\mu, q}$ is sectorial, we see $\Vert \tilde{A}_{\mu, q}^{m} e^{- t \tilde{A}_{\mu, q}} f \Vert_{L^q(\Omega)^2} \leq t^{ - m} e^{- \mu_\ast t} C \Vert f \Vert_{L^q(\Omega)^2}$ for $m = 0,1$.
    Interpolating between them, we obtain the estimate.
\end{proof}

\begin{proposition} \label{prop_H1_t_bound_linear_case}
    Let $1 < p, q < \infty$, $\eta > 1/p$, and $T>0$.
    Let $V \in \mathbb{E}_{1, \eta, p, q, T}$ be the solution to
    \begin{equation} \label{eq_linear_perturbed_hydrostatic_Stokes_eq_with_zero_id}
        \begin{split}
            \begin{aligned}
                \partial_t V + \tilde{A}_{\mu, q} V
                & = P F \\
                V(0)
                & = 0
            \end{aligned}
        \end{split}
    \end{equation}
    for $F \in H^{1,p}_\eta(0, T; L^q(\Omega)^2)$ satisfying $t \partial_t F \in \mathbb{E}_{0, \eta, p, q, T}$ and $V_0 \in X_{1/q, p, q}$.
    Then there exists a $T$-independent constant $C > 0$ such that
    \begin{align*}
        & \Vert
            t \partial_t V
        \Vert_{H^{1, p}_\eta(0, T; L^q(\Omega)^2)}
        + \Vert
            t \partial_t V
        \Vert_{L^p_\eta(0, T; D(\tilde{A}_{\mu, q}))} \\
        & \leq C \left(
            \Vert
                V
            \Vert_{\mathbb{E}_{1, \eta, p, q, T}}
            + \Vert
                F
            \Vert_{\mathbb{E}_{0, \eta, p, q, T}}
            + \Vert
                t \partial_t F
            \Vert_{\mathbb{E}_{0, \eta, p, q, T}}
        \right).
    \end{align*}
\end{proposition}
\begin{proof}
    Since $F \in H^{1, p}_\eta(0, T; L^q(\Omega)^2)$, there exists a unique solution $V$ to (\ref{eq_linear_perturbed_hydrostatic_Stokes_eq_with_zero_id}) such that
    \begin{align*}
        \Vert
            V
        \Vert_{H^{2, p}_\eta(0, T; L^q(\Omega)^2) \cap H^{1, p}_\eta(0, T; L^q(\Omega)^2)}
        \leq C \Vert
            F
        \Vert_{H^{1, p}_\eta(0, T; L^q(\Omega)^2)}
    \end{align*}
    for some constant $C>0$.
    On the other hand, since $V \in \mathbb{E}_{1, \eta, p, q, T}$ and $t \partial_t F \in \mathbb{E}_{0, \eta, p, q, T}$, there exists a unique solution $\Phi \in \mathbb{E}_{1, \eta, p, q, T}$ to
    \begin{align} \label{eq_tdtV}
        \begin{split}
            \partial_t \Phi + \tilde{A}_{\mu, q} \Phi
            & = - \tilde{A}_{\mu, q} V
            + PF
            + t \partial_t P F, \\
            V(0)
            & = 0,
        \end{split}
    \end{align}
    such that
    \begin{align} \label{eq_estimate_Phi}
        \Vert
            \Phi
        \Vert_{\mathbb{E}_{1, \eta, p, q, T}}
        \leq C \left(
            \Vert
                V
            \Vert_{\mathbb{E}_{1, \eta, p, q, T}}
            + \Vert
                F
            \Vert_{\mathbb{E}_{1, \eta, p, q, T}}
            + \Vert
                t \partial_t F
            \Vert_{\mathbb{E}_{1, \eta, p, q, T}}
            \right),
    \end{align}
    where $C$ is independent of $T$.
    The vector field $t \partial_t V$ also satisfies (\ref{eq_tdtV}).
    By the uniqueness, we find $\Phi = t \partial_t V$.
    Therefore, $t \partial_t V$ satisfies (\ref{eq_estimate_Phi}).
\end{proof}

\subsection{Local well-posedness}
We establish the local well-posedness and global well-posedness for small data to (\ref{eq_diff}) in the maximal regularity settings.
We begin with the nonlinear estimates, see Giga et al. \cite{GigaGriesHieberHusseinKashiwabara2017_analiticity}.
\begin{proposition} \label{prop_bilinear_estimate_sobolev}
    There exits a constant $C>0$ such that
    \begin{align*}
        & \Vert
            v_1 \cdot \nabla_H v_2
        \Vert_{H^{s, q}(\Omega)^2}
        + \left \Vert
            \int_{-l}^{x_3}
                \mathrm{div}_H v_1
            dz \,
            \partial_3 v_2
        \right \Vert_{{H^{s, q}(\Omega)^2}} \\
        & \leq C \Vert
            v_1
        \Vert_{H^{s + 1 + \frac{1}{q}, q}(\Omega)^2}
        \Vert
            v_2
        \Vert_{H^{s + 1 + \frac{1}{q}, q}(\Omega)^2}
    \end{align*}
    for all $v_1, v_2 \in H^{s + 1 + \frac{1}{q}, q}(\Omega)^2$.
\end{proposition}

\begin{proposition} \label{prop_bilinear_estimate_bilinear_maximal_regularity}
    Let $1 < p, q < \infty$ satisfying $\eta = \frac{1}{p} + \frac{1}{q} \leq 1$.
    Let $0 \leq \gamma \leq \mu_\ast$.

    (i) There exits a $T$-independent constant $C>0$ such that
    \begin{align} \label{eq_bilinear_estimate_bilinear_maximal_regularity_1}
        \begin{split}
            & \Vert
                \left(
                    f \cdot \nabla_H g
                \right)
            \Vert_{\mathbb{E}_{0, \eta, p, q, T, e^{\gamma t}}}
            + \left \Vert
                \left(
                    \int_{-l}^{x_3}
                        \mathrm{div}_H f
                    dz \,
                    \partial_3 g
                \right)
            \right \Vert_{\mathbb{E}_{0, \eta, p, q, T, e^{\gamma t}}} \\
            & \leq C \Vert
                f
            \Vert_{\mathbb{E}_{1, \eta, p, q, T, e^{\gamma t}}}
            \Vert
                g
            \Vert_{\mathbb{E}_{1, \eta, p, q, T, e^{\gamma t}}}
        \end{split}
    \end{align}
    for $f, g \in \mathbb{E}_{1, \eta, p, q, T, e^{\gamma t}}$ satisfying $f|_{t = 0} = 0, g|_{t = 0} = 0$.
    
    (ii) Assume $\tilde{f} = e^{ - \mathcal{A}_1 t} f_0 + f$ and $\tilde{g} = e^{ - \mathcal{A}_2 t} g_0 + g$ for $f_0, g_0 \in X_{1/q, p, q}$ and $\mathcal{A}_j = A_q, \tilde{A}_{\mu, q}$.
    For $\gamma \leq \mu_\ast$, if not $\mathcal{A}_1 = \mathcal{A}_2 = A_q$, it follows that
    \begin{align} \label{eq_bilinear_estimate_bilinear_maximal_regularity_2}
        \begin{split}
            & \Vert
                \tilde{f} \cdot \nabla_H \tilde{g}
            \Vert_{\mathbb{E}_{0, \eta, p, q, T, e^{\gamma t}}}
            + \Vert
                \int_{-l}^{x_3}
                    \mathrm{div}_H \tilde{f}
                dz \,
                \partial_3 \tilde{g}
            \Vert_{\mathbb{E}_{0, \eta, p, q, T, e^{\gamma t}}} \\
            & \leq C_1 \Vert
                f_0
            \Vert_{X_{1/q, p, q}}
            \Vert
                g_0
            \Vert_{X_{1/q, p, q}}
            + C_2 \Vert
                f_0
            \Vert_{X_{1/q, p, q}}
            \Vert
                g
            \Vert_{\mathbb{E}_{1, \eta, p, q, T, e^{\gamma t}}} \\
            & + C_3 \Vert
                g_0
            \Vert_{X_{1/q, p, q}}
            \Vert
                f
            \Vert_{\mathbb{E}_{1, \eta, p, q, T, e^{\gamma t}}}
            + C_4 \Vert
                f
            \Vert_{\mathbb{E}_{1, \eta, p, q, T, e^{\gamma t}}}
            \Vert
                g
            \Vert_{\mathbb{E}_{1, \eta, p, q, T, e^{\gamma t}}}.
        \end{split}
    \end{align}
\end{proposition}
\begin{proof}
    The Hardy inequality is such
    \begin{align*}
        \Vert
            \phi
        \Vert_{L^p_{1 + \theta - s}(0, T)}
        \leq C \Vert
            \phi
        \Vert_{H^{s, p}_{1 + \theta}(0, T)}
    \end{align*}
    for $s \in [0, 1]$, $1/p < \theta < 1$, and $\phi \in H^{s, p}_{1 + \theta}(0, T)$ satisfying $\phi |_{t = 0} = 0$.
    Combining with the Hardy inequality, Proposition \ref{prop_bilinear_estimate_sobolev}, and the mixed derivative theorem, we find
    \begin{align*}
        & \Vert
            f \cdot \nabla_H g
        \Vert_{\mathbb{E}_{0, \eta, p, q, T, e^{\gamma t}}}
        + \left \Vert
            \int_{-l}^{x_3}
                \mathrm{div}_H f
            dz \,
            \partial_3 g
        \right \Vert_{\mathbb{E}_{0, \eta, p, q, T, e^{\gamma t}}} \\
        & \leq C \Vert
            e^{\gamma t} f
        \Vert_{L^{2p}_{(1 + \eta)/2}(0, T; H^{1 + 1/q, q}(\Omega)^2)}
        \Vert
            g
        \Vert_{L^{2p}_{(1 + \eta)/2}(0, T; H^{1 + 1/q, q}(\Omega)^2)}.
    \end{align*}
    Applying the Hardy inequality again, we have
    \begin{align*}
        & \Vert
            e^{\gamma t} f
        \Vert_{L^{2p}_{(1 + \eta)/2}(0, T; H^{1 + 1/q, q}(\Omega)^2)}\\
        & \leq C \Vert
            e^{\gamma t} f
        \Vert_{H^{1/2p + (1 - \eta)/2, p}_{\eta}(0, T; H^{1 + 1/q, q}(\Omega)^2)}
        \leq C \Vert
            f
        \Vert_{\mathbb{E}_{0, \eta, p, q, T, e^{\gamma t}}}.
    \end{align*}
    The same inequality holds for $g$, then we obtain (\ref{eq_bilinear_estimate_bilinear_maximal_regularity_1}).
    By the choice of the time weight index $\eta$, the embedding constants are independent of $T$.
    We use Proposition 3.4.3 in \cite{PrussSimonett2016} to see
    \begin{align} \label{eq_bound_for_linear_part_in_mixed_space}
        \Vert
            e^{\gamma t} e^{- t \tilde{A}_{\mu, q}} f_0
        \Vert_{L^{2p}_{(1 + \eta)/2}(0, T; H^{1 + 1/q, q}(\Omega)^2)}
        \leq C \Vert
            f_0
        \Vert_{D_{\tilde{A}_{\mu, q}}(1 - \eta, 2p)}
        \leq C \Vert
            f_0
        \Vert_{X_{1/q, p, q}}.
    \end{align}
    The embedding constants are independent of $T$.
    The same inequality holds for $e^{- t \tilde{A}_{\mu, q}} g_0$.
    Combining with (\ref{eq_bilinear_estimate_bilinear_maximal_regularity_1}) and the above estimates, we obtain (\ref{eq_bilinear_estimate_bilinear_maximal_regularity_2}).
\end{proof}

\begin{remark} \label{rmk_smallness_constant_for_small_T}
    The estimate (\ref{eq_bound_for_linear_part_in_mixed_space}) implies that the constants $C_j$ for $j = 1, 2, 3$ can be bounded small if we take $T$ sufficiently small since the norm $\Vert \cdot \Vert_{L^{2p}_{(1 + \eta)/2}(0, T; H^{1 + 1/q, q}(\Omega)^2)}$ is a integral norm with respect to $t$.
\end{remark}

\begin{proposition} \label{prop_local_wellposedness_maximal_regularity}
    Let $0 < p, q < 1$ satisfying $\frac{1}{p} + \frac{1}{q} \leq 1$.
    Let $\eta = \frac{1}{p} + \frac{1}{q}$ and $T>0$.
    Let $V_0 \in X_{1/q, p, q}$ and $F \in \mathbb{E}_{0, \eta, p, q, T}$.
    Assume $v \in \mathbb{E}_{1, \eta, p, q, T, e^{\gamma_1 t}}$ be the solution with initial data $v_0 \in X_{1/q, p, q}$ and an external force $f \in \mathbb{E}_{0, \eta, p, q, T, e^{\gamma_0 t}}$ to (\ref{eq_primitive}) such that
    \begin{align*}
        \Vert
            \partial_t v
        \Vert_{\mathbb{E}_{0, \eta, p, q, T, e^{\gamma_1 t}}}
        + \Vert
            A_q v
        \Vert_{\mathbb{E}_{0, \eta, p, q, T, e^{\gamma_1 t}}}
        \leq C
    \end{align*}
    for some $T$-independent constants $C>0$ and $\gamma_0 \geq \gamma_1 > 0$.

    (i) There exit $T_0>0$, if $T \leq T_0$, there exits a unique solution $V \in \mathbb{E}_{1, \eta, p, q, T}$ to (\ref{eq_diff}) with initial data $V_0$ such that
    \begin{align*}
        \Vert V \Vert_{\mathbb{E}_{1, \eta, p, q, T}}
        \leq C
    \end{align*}
    for some constant $C>0$.

    (ii) There exists $\varepsilon_0 > 0$, if $\Vert v \Vert_{\mathbb{E}_{1, \eta, p, q, T, e^{\mu_\ast t}}}, \Vert V_0 \Vert_{X_{1/q, p, q, T, e^{\mu_\ast t}}}, \Vert F \Vert_{\mathbb{E}_{0, \eta, p, q, T, e^{\mu_\ast t}}} \leq \varepsilon_0$, there exits a unique solution $V \in \mathbb{E}_{1, \eta, p, q, T, e^{\mu_\ast t}}$ to (\ref{eq_diff}) with initial data $V_0$ such that
    \begin{align*}
        \Vert V \Vert_{\mathbb{E}_{1, \eta, p, q, T, e^{\mu_\ast t}}}
        \leq C
    \end{align*}
    for some constant $C>0$, where the case $T = \infty$ is included.

    (iii) If $F \in H^{1, p}_\eta (0, T; L^q(\Omega)^2)$ and $t \partial_t F \in \mathbb{E}_{1, \eta, p, q, e^{\gamma t}}$, there exits $T_0 > 0$, if $T \leq T_0$,
    \begin{align*}
        \Vert
            V
        \Vert_{\mathbb{E}_{1, \eta, p, q, e^{\gamma t}}}
        + \Vert
            t \partial_t V
        \Vert_{\mathbb{E}_{1, \eta, p, q, e^{\gamma t}}}
        \leq C,
    \end{align*}
    holds for some constant $C>0$.
\end{proposition}
\begin{proof}
    We prove (i) and (ii).
    We set
    \begin{align*}
        \mathcal{N}(V)
        = - U \cdot \nabla V
        + u \cdot \nabla V
        + V \cdot \nabla v.
    \end{align*}
    Let $\Phi^\prime \in \mathbb{E}_{1, \eta, p, q, T}$ be the solution to
    \begin{align*}
        \partial_t \Phi^\prime
        + \tilde{A}_{\mu, q} \Phi^\prime
        & = 0, \\
        \Phi^\prime (0)
        & = V_0.
    \end{align*}
    We put $\Phi = (\Phi^\prime, \Phi_3)$ for $\Phi_3 = \int_{-l}^{x_3} \mathrm{div}_H \Phi^\prime dz$.
    By the maximal regularity of $\tilde{A}_{\mu, q}$, $\Phi$ satisfies
    \begin{align*}
        \Vert
            e^{\mu_\ast t} \Phi^\prime
        \Vert_{\mathbb{E}_{1, \eta, p, q, T}}
        \leq C \Vert
            V_0
        \Vert_{X_{1/q, p, q}}
    \end{align*}
    for some $T$-independent constant $C > 0$.
    Let $\tilde{V} = V - \Phi^\prime$, then $\tilde{V}$ satisfies
    \begin{align} \label{eq_for_V_tilde_abstract_formulations}
        \begin{split}
            \partial_t \tilde{V}
            + \tilde{A}_{\mu, q} \tilde{V}
            & = \tilde{\mathcal{N}}(\tilde{V}), \\
            \tilde{V} (0)
            & = 0,
        \end{split}
    \end{align}
    where
    \begin{align*}
        \tilde{\mathcal{N}}(\tilde{V})
        & = - \tilde{U} \cdot \nabla \tilde{V}
        - \Phi \cdot \nabla \tilde{V}
        - \tilde{U} \cdot \nabla \Phi^\prime
        - \Phi \cdot \nabla \Phi^\prime \\
        & + u \cdot \nabla V
        + u \cdot \nabla \Phi^\prime
        + V \cdot \nabla v
        + \Phi \cdot \nabla v
        + F.
    \end{align*}
    By Proposition \ref{prop_bilinear_estimate_bilinear_maximal_regularity}, we find
    \begin{align} \label{eq_estimate_mathcalN_in_E0}
        \begin{split}
            \Vert
                e^{\tilde{\mu} t} \tilde{\mathcal{N}}(\tilde{V})
            \Vert_{\mathbb{E}_{0, \eta, p, q, T}}
            & \leq C_2 \Vert
                e^{\tilde{\mu} t} \tilde{V}
            \Vert_{\mathbb{E}_{1, \eta, p, q, T}}^2 \\
            & + C_1 \left(
                \Vert
                    e^{\tilde{\mu} t} \Phi^\prime
                \Vert_{\mathbb{E}_{1, \eta, p, q, T}}
                + \Vert
                    e^{\tilde{\mu} t} v
                \Vert_{\mathbb{E}_{1, \eta, p, q, T}}
            \right)
            \Vert
                e^{\tilde{\mu} t} \tilde{V}
            \Vert_{\mathbb{E}_{1, \eta, p, q, T}} \\
            & + C_0 \left(
                \Vert
                    e^{\tilde{\mu} t} \Phi^\prime
                \Vert_{\mathbb{E}_{1, \eta, p, q, T}}^2
                + \Vert
                    e^{\tilde{\mu} t} \Phi^\prime
                \Vert_{\mathbb{E}_{1, \eta, p, q, T}}
                \Vert
                    e^{\tilde{\mu} t} v
                \Vert_{\mathbb{E}_{1, \eta, p, q, T}}
            \right. \\
            & \left.
                \quad \quad \quad 
                + \Vert
                    e^{\tilde{\mu} t} v
                \Vert_{\mathbb{E}_{1, \eta, p, q, T}}^2
                + \Vert
                    e^{\tilde{\mu} t} F
                \Vert_{\mathbb{E}_{1, \eta, p, q, T}}
            \right).
        \end{split}
    \end{align}
    We write $\mathcal{T}: \mathbb{E}_{0, \eta, p, q, T, e^{\tilde{\mu}t}} \rightarrow \mathbb{E}_{1, \eta, p, q, T, e^{\tilde{\mu}t}}$ to denote the solution operator $\mathcal{T} g = X$ for
    \begin{align*}
        \partial_t X
        + \tilde{A}_{\mu, q} X
        & = g, \\
        X (0)
        & = 0.
    \end{align*}
    If we take $T$ small or take $\Vert V_0 \Vert_{X_{1/q, p, q}}$, $\Vert v_0 \Vert_{X_{1/q, p, q}}$, $\Vert F \Vert_{\mathbb{E}_{0, \eta, p, q, T, e^{\tilde{\mu}t}}}$ small, then by Remark \ref{rmk_smallness_constant_for_small_T} we find that $\Vert e^{\tilde{\mu} t} \Phi^\prime \Vert_{\mathbb{E}_{1, \eta, p, q, T, e^{\tilde{\mu} t}}}$, $\Vert e^{\tilde{\mu} t} v \Vert_{\mathbb{E}_{1, \eta, p, q, T, e^{\tilde{\mu} t}}}$, and $\Vert e^{\tilde{\mu} t} F \Vert_{\mathbb{E}_{0, \eta, p, q, T, e^{\tilde{\mu} t}}}$ can be sufficiently small.
    Combining with these observations and the estimate above, we see there exists small $R>0$ such that
    \begin{align*}
        \mathcal{T}\tilde{\mathcal{N}}
        : \mathbb{E}_{1, \eta, p, q, e^{\tilde{\mu}t}}(T) \rightarrow \mathbb{E}_{1, \eta, p, q, e^{\tilde{\mu}t}}(T),
        \quad \Vert \tilde{V} \Vert_{\mathbb{E}_{1, \eta, p, q, T, e^{\tilde{\mu} t}}} \leq R \quad
        \text{for small $R>0$},
    \end{align*}
    is a self-mapping.
    Let $\tilde{V}_1, \tilde{V}_2 \in \mathbb{E}_{1, \eta, p, q, T}$ satisfying
    \begin{align*}
        \Vert \tilde{V}_1 \Vert_{\mathbb{E}_{1, \eta, p, q, T}},
        \Vert \tilde{V}_2 \Vert_{\mathbb{E}_{1, \eta, p, q, T}}
        \leq R.
    \end{align*}
    Using the formula
    \begin{align*}
        &\tilde{\mathcal{N}}(\tilde{V}_1) - \tilde{\mathcal{N}}(\tilde{V}_2) \\
        &= - (\tilde{U}_1 - \tilde{U}_2) \cdot \nabla \tilde{V}_1
        - \tilde{U}_2 \cdot \nabla (\tilde{V}_1 - \tilde{V}_2) \\
        &- \Phi \cdot \nabla (\tilde{V}_1 - \tilde{V}_2)
        - (\tilde{V}_1 - \tilde{V}_2) \cdot \nabla \Phi^\prime \\
        &+ u \cdot \nabla (\tilde{V}_1 - \tilde{V}_2)
        + (\tilde{V}_1 - \tilde{V}_2) \cdot \nabla v
    \end{align*}
    and Proposition \ref{prop_bilinear_estimate_bilinear_maximal_regularity}, we find
    \begin{align} \label{eq_contractivity_of_quadratic_terms_maximal_regularity_space}
        \begin{split}
            & \Vert
                e^{\tilde{\mu} t} \left(
                    \tilde{\mathcal{N}}(\tilde{V}_1) - \tilde{\mathcal{N}}(\tilde{V}_2)
                \right)
            \Vert_{\mathbb{E}_{0, \eta, p, q, T}} \\
            & \leq C \left(
                \Vert
                    e^{\tilde{\mu} t} \tilde{V}_1
                \Vert_{\mathbb{E}_{1, \eta, p, q, T}}
                + \Vert
                    e^{\tilde{\mu} t} \tilde{V}_2
                \Vert_{\mathbb{E}_{1, \eta, p, q, T}}
                + \Vert
                    e^{\tilde{\mu} t} v
                \Vert_{\mathbb{E}_{1, \eta, p, q, T}}
                + \Vert
                    e^{\tilde{\mu} t} \Phi^\prime
                \Vert_{\mathbb{E}_{1, \eta, p, q, T}}
            \right) \\
            &  \quad \times \Vert
                e^{\tilde{\mu} t} \left(
                    \tilde{V}_1 - \tilde{V}_2
                \right)
            \Vert_{\mathbb{E}_{1, \eta, p, q, T}}.
        \end{split}
    \end{align}
    If we take $\Vert \Phi^\prime \Vert_{\mathbb{E}_{1, \eta, p, q, T, e^{\tilde{\mu} t}}}$ and $\Vert v \Vert_{\mathbb{E}_{1, \eta, p, q, T, e^{\tilde{\mu} t}}}$, and $R$ sufficiently small, then $\mathcal{T}\tilde{\mathcal{N}}$ is a contraction mapping on $\mathbb{E}_{1, \eta, p, q, e^{\tilde{\mu}t}}(T)$.
    Banach's fixed point theorem implies that there exists a unique solution $\tilde{V} \in \mathbb{E}_{1, \eta, p, q, T, e^{\tilde{\mu}t}}$ to (\ref{eq_for_V_tilde_abstract_formulations}) such that $\Vert \tilde{V} \Vert_{\mathbb{E}_{1, \eta, p, q, T, e^{\tilde{\mu}t}}} \leq R$.
    The solution to (\ref{eq_diff}) is given by $V = \tilde{V} + \Phi^\prime$.
    The solution $V$ satisfies
    \begin{align*}
        \Vert
            e^{\tilde{\mu} t} V
        \Vert_{\mathbb{E}_{1, \eta, p, q, T, e^{\tilde{\mu}t}}}
        \leq C.
    \end{align*}
    Suppose there exists two solutions $V_1, V_2$ satisfying $V_1 \neq V_2$ and $\Vert V_j \Vert_{\mathbb{E}_{1, \eta, p, q, T, e^{\tilde{\mu}t}}} \leq R$.
    The estimate (\ref{eq_contractivity_of_quadratic_terms_maximal_regularity_space}) implies if $T$ is small or $R$, $v$, and $\Phi^\prime$ are small, we see
    \begin{align*}
        \Vert
            e^{\tilde{\mu} t} (V_2 - V_2)
        \Vert_{\mathbb{E}_{1, \eta, p, q, T, e^{\tilde{\mu}t}}}
        \leq \varepsilon \Vert
            e^{\tilde{\mu} t} (V_2 - V_2)
        \Vert_{\mathbb{E}_{1, \eta, p, q, T, e^{\tilde{\mu}t}}},
    \end{align*}
    for some $\varepsilon \in (0,1)$.
    This implies $V_1 = V_2$ and a contradiction.
    Therefore, $V$ is unique in $\mathbb{E}_{1, \eta, p, q, T, e^{\tilde{\mu}t}}$.
    We proved (i) and (ii).

    We show (iii).
    We find from Proposition \ref{prop_bilinear_estimate_bilinear_maximal_regularity}
    \begin{align} \label{eq_bound_partial_t_N_tilde_V_in_maximal_regularity_space}
        \begin{split}
            & \Vert
                e^{\tilde{\mu} t} t \partial_t \tilde{\mathcal{N}}(\tilde{V})
            \Vert_{\mathbb{E}_{0, \eta, p, q, T}} \\
            & \leq C_2 \Vert
                e^{\tilde{\mu} t} t \partial_t \tilde{V}
            \Vert_{\mathbb{E}_{1, \eta, p, q, T}}
            \Vert
                e^{\tilde{\mu} t} \tilde{V}
            \Vert_{\mathbb{E}_{1, \eta, p, q, T}} \\
            & + C_{11} \left(
                \Vert
                    e^{\tilde{\mu} t} \Phi^\prime
                \Vert_{\mathbb{E}_{1, \eta, p, q, T}}
                + \Vert
                    e^{\tilde{\mu} t} v
                \Vert_{\mathbb{E}_{1, \eta, p, q, T}}
            \right)
            \Vert
                e^{\tilde{\mu} t} t \partial_t \tilde{V}
            \Vert_{\mathbb{E}_{1, \eta, p, q, T}} \\
            & + C_{12} \left(
                \Vert
                    e^{\tilde{\mu} t} t \partial_t \Phi^\prime
                \Vert_{\mathbb{E}_{1, \eta, p, q, T}}
                + \Vert
                    e^{\tilde{\mu} t} t \partial_t v
                \Vert_{\mathbb{E}_{1, \eta, p, q, T}}
            \right)
            \Vert
                e^{\tilde{\mu} t} \tilde{V}
            \Vert_{\mathbb{E}_{1, \eta, p, q, T}} \\
            & + C_{01} (
                \Vert
                    e^{\tilde{\mu} t} t \partial_t \Phi^\prime
                \Vert_{\mathbb{E}_{1, \eta, p, q, T}}
                + \Vert
                    e^{\tilde{\mu} t} \Phi^\prime
                \Vert_{\mathbb{E}_{1, \eta, p, q, T}}
            )
            (
                \Vert
                    e^{\tilde{\mu} t} t \partial_t v
                \Vert_{\mathbb{E}_{1, \eta, p, q, T}}
                + \Vert
                    e^{\tilde{\mu} t} v
                \Vert_{\mathbb{E}_{1, \eta, p, q, T}}
            ) \\
            & + C_{02} \Vert
                e^{\tilde{\mu} t} t \partial_t F
            \Vert_{\mathbb{E}_{1, \eta, p, q, T}}.
        \end{split}
    \end{align}
    for some constants $C_2>0, C_{ij}>0$ ($i = 0,1, j=1, 2$).
    Note that since $\partial_t \Phi^\prime = \partial_t e^{t \tilde{A}_{\mu, q}/2} e^{t \tilde{A}_{\mu, q}/2} V_0$, we find from Proposition \ref{prop_pointwise_estimate_semigroup} that
    \begin{align*}
        \Vert
            t \partial_t \Phi
        \Vert_{\mathbb{E}_{1, \eta, p, q, T}}
        \leq C \Vert
            V_0
        \Vert_{X_{1/q, p, q}}
    \end{align*}
    for some constant $C > 0$.
    By Propositions \ref{prop_H1_t_bound_linear_case}, we have
    \begin{align} \label{eq_bound_for_t_partialt_T_N}
        \begin{split}
            & \Vert
                t \partial_t \mathcal{T} \mathcal{N} (\tilde{V})
            \Vert_{\mathbb{E}_{1, \eta, p, q, e^{\tilde{\mu}t }}(T)} \\
            & \leq C \left(
                \Vert
                    \mathcal{T} \mathcal{N} (\tilde{V})
                \Vert_{\mathbb{E}_{1, \eta, p, q, e^{\tilde{\mu}t }}(T)}
                + \Vert
                \mathcal{N} (\tilde{V})
                \Vert_{\mathbb{E}_{0, \eta, p, q, e^{\tilde{\mu}t }}(T)}
                + \Vert
                    t \partial_t \mathcal{N} (\tilde{V})
                \Vert_{\mathbb{E}_{0, \eta, p, q, e^{\tilde{\mu}t }}(T)}
            \right).
        \end{split}
    \end{align}
    We define the mixed Banach space $\mathcal{U}_{\eta, p, q, T}$ by
    \begin{align*}
        \mathcal{U}_{\eta, p, q, T}
        & = \Set{
            \varphi \in H^{2, p}_{\eta} (0, T; L^q(\Omega)^2) \cap H^{1, p}_{\eta} (0, T; D(\tilde{A}_{\mu, q}))
        }{
            \Vert
            \varphi
            \Vert_{\mathcal{U}_{\eta, p, q, T}}
            < \infty
        } \\
        \Vert
            \varphi
        \Vert_{\mathcal{U}_{\eta, p, q, T}}
        & := \alpha \Vert
            \varphi
        \Vert_{\mathbb{E}_{0, \eta, p, q, T, e^{\tilde{\mu}t}}}
        + \Vert
            t \partial_t \varphi
        \Vert_{\mathbb{E}_{0, \eta, p, q, T, e^{\tilde{\mu}t}}}
    \end{align*}
    for some constant $\alpha >0$. 
    This $\alpha$ can be taken so large that it absorb the constant $C$ in (\ref{eq_bound_for_t_partialt_T_N}).
    Using the estimate (\ref{eq_bound_partial_t_N_tilde_V_in_maximal_regularity_space}), if we take $T>0$ sufficiently small, then we see $\mathcal{T}: \mathcal{U}_{\eta, p, q, T} \rightarrow \mathcal{U}_{\eta, p, q, T}$ can be a self-mapping.
    By the same way to prove (i), we find that $\mathcal{T}: \mathcal{U}_{\eta, p, q, T} \rightarrow \mathcal{U}_{\eta, p, q, T}$ is a contraction mapping for small $T$.
    Then Banach's fixed point theorem implies the conclusion of (iii).
\end{proof}

\begin{proposition} \label{prop_local_wellposedness_maximal_regularity_DAeq}
    Under the same assumptions for $p, q, \eta, V_0, v_0, v, F$, there exits $T_0>0$, if $T \leq T_0$, there exits a unique solution $V \in \mathbb{E}_{1, \eta, p, q}$ to (\ref{eq_nudging}) with initial data $V_0$ such that $\Vert V \Vert_{\mathbb{E}_{1, \eta, p, q, T}} \leq C$ for some constant $C>0$.
\end{proposition}
\begin{proof}
    The proof is based on Proposition \ref{prop_H_infty_for_A_plus_J} and same kind of contraction argument as Proposition \ref{prop_local_wellposedness_maximal_regularity}.
    Since the proof is quite similar to that of Proposition \ref{prop_local_wellposedness_maximal_regularity}, we omit the detail.
\end{proof}

\begin{remark}
    We know the global well-posedness of the primitive equations, local well-posedness of (\ref{eq_diff}), and the uniqueness of the solution to (\ref{eq_nudging}) in $\mathbb{E}_{1, \eta, p, q, T}$ for some $T>0$.
    Therefore, we can conclude that, if we establish the global well-posedness of (\ref{eq_diff}) in $\mathbb{E}_{1, \eta, p, q, \infty}$, the equations (\ref{eq_nudging}) are also globally well-posed in $\mathbb{E}_{1, \eta, p, q, \infty}$.
\end{remark}

\subsection{Global well-posedness and asymptotic behavior}
We assume a priori bounds of the solution $V$ to (\ref{eq_diff}) such that
\begin{gather} \label{eq_a_priori_bounds}
    \begin{split}
        \Vert
            V (t)
        \Vert_{L^2(\Omega)^2}^2
        \leq e^{- ct} \Vert
            V_0
        \Vert_{L^2(\Omega)^2}^2
        + C \delta \int_0^t
            e^{- c (t - s)}
            \Vert
                \nabla f (s)
            \Vert_{L^2(\Omega)^2}^{2}
        ds, \\
        \Vert
            V (t)
        \Vert_{L^2(\Omega)^2}^2
        + \int_0^t
            \Vert
                \nabla V(s)
            \Vert_{L^2(\Omega)^2}^2
        ds
        \leq \Vert
            V_0
        \Vert_{L^2(\Omega)^2}^2
        + C \delta \int_0^t
            \Vert
                \nabla f (s)
            \Vert_{L^2(\Omega)^2}^{2}
        ds. \\
        \Vert
            V (t)
        \Vert_{H^2(\Omega)^2}^2
        \leq \phi(t)
    \end{split}
\end{gather}
for some $c>0$ and $\phi \in BC(\varepsilon, T]$ for some small $\varepsilon > 0$.
These a priori estimate is proved in Section \ref{sec_a_priori_estimate}.
\begin{lemma} \label{lem_pre_of_main_theorem}
    Under the same assumption of Theorem \ref{thm_main_thoerem}.
    We assume a priori bounds (\ref{eq_a_priori_bounds}).
    Then there exit $\mu_0>0$, $\delta_0 >0$, if $\mu \geq \mu_0$ and $\delta \leq \delta_0$, there exits a unique solution $V \in C(0, T; X_{1/q, p, q})$ to (\ref{eq_nudging}) such that
    \begin{align*}
        \Vert
            e^{\mu_\ast t} \partial_t V
        \Vert_{L^p_\eta(0, T; L^q(\Omega)^2)}
        + \Vert
            e^{\mu_\ast t} V
        \Vert_{L^p_\eta(0, T; H^{2,q}(\Omega)^2)}
        \leq C,
    \end{align*}
    for some constants $\mu_\ast > 0$ and $C>0$, where the case $T = \infty$ is included.
    Moreover, $V$ is exponentially stable in the sense
    \begin{align*}
        \Vert
            V (t)
        \Vert_{X_{1/q, p, q}}
        \leq C e^{ - \mu_\ast t}.
    \end{align*}
\end{lemma}
\begin{proof}
    We begin by the case $p = q = 2$.
    By (\ref{eq_a_priori_bounds}) and Proposition \ref{prop_local_wellposedness_maximal_regularity}(i), we have a global solution $V \in \mathbb{E}_{1, 1, 2, 2, T}$ for $T < \infty$.
    We consider the case $T = \infty$.
    Since $\int_0^\infty \Vert \nabla V(s) \Vert_{L^2(\Omega)^2}^2 ds$ is bounded, there exits $t_0 > 0$ and small $\varepsilon_0 > 0$ such that
    \begin{align*}
        \Vert V(t_0) \Vert_{H^1(\Omega)^2}
        \leq \varepsilon_0.
    \end{align*}
    Applying Proposition \ref{prop_local_wellposedness_maximal_regularity}(ii), we find the solution $V$ with initial data $V(t_0)$ belongs to $\mathbb{E}_{1, 2, 2, \infty, e^{\mu_\ast t}}$ and satisfies $\Vert V(t) \Vert_{X_{1/q, p, q}} = O(e^{- \mu_\ast t})$.
    We finished the proof for the case $p = q = 2$.
    Since
    \begin{align*}
        \int_0^\infty
            e^{2 \mu_\ast s} \Vert
                V(s)
            \Vert_{D(\tilde{A}_{\mu, 2})}^2
        ds
        < \infty,
    \end{align*}
    if we take $t_1 >0$ sufficiently large, there exits small $\varepsilon_1 > 0$ such that $\Vert V(t_1) \Vert_{D(\tilde{A}_{\mu, 2})} \leq \varepsilon_1$.
    Using this bound, the assumption $F \in H^1_{loc}(0, \infty; L^2(\Omega)^2)$, and Proposition \ref{prop_local_wellposedness_maximal_regularity}(ii, iii), we see $V \in BC(0, \infty; D(\tilde{A}_{\mu, 2}))$ and $\Vert V(t) \Vert_{H^2(\Omega)^2}$ is small for some $t>0$.

    Let $T^\prime$ be the maximal existence time of solution $V$ in $\mathbb{E}_{1, \eta, p, q, T^\prime}$ given by Proposition \ref{prop_local_wellposedness_maximal_regularity}(i).
    Proposition \ref{prop_local_wellposedness_maximal_regularity} implies $V(t) \in D(\tilde{A}_{\mu, q}) \hookrightarrow H^{2,q}_{\overline{\sigma}}(\Omega)$ for $t < T^\prime$.
    When $q \geq 6/5$, we find from the Sobolev embedding that $V(t) \in H^1(\Omega)^2$.
    For the case $1 < q < 6/5$, we use $F \in L^2(0, T; L^2(\Omega)^2)$ and the bootstrapping argument such as \cite{HieberHusseingKashiwabara2016} to improve the regularity as $V(t) \in H^1(\Omega)^2$ for $\frac{T^\prime}{2} \leq t < T^\prime$.
    Therefore, by the result when $p = q = 2$, we can obtain the global $H^1$-solution such that $V \in \mathbb{E}_{1/q, p, q, T^\prime} \cap BC(T^\prime/2, \infty; D(\tilde{A}_{\mu, 2}))$ for all any $1 < q < \infty$.
    Since $H^{2,2}(\Omega)^2 \hookrightarrow X_{1/q, p, q}$, we see $V(t) \in X_{1/q, p, q}$ for all $t > 0$.
    If we take $t_2>0$ sufficiently large, then we find $\Vert V(t_2) \Vert_{H^2(\Omega)^2} \leq \varepsilon_2$ for some small $\varepsilon_2 > 0$.
    By Proposition \ref{prop_local_wellposedness_maximal_regularity} (ii), we deduce
    \begin{align*}
        \Vert
            e^{\mu_\ast t} \partial_t V
        \Vert_{L^p_\eta(0, \infty; L^q(\Omega)^2)}
        + \Vert
            e^{\mu_\ast t} V
        \Vert_{L^p_\eta(0, \infty; H^{2,q}(\Omega)^2)}
        \leq C
    \end{align*}
    for some constant $C>0$, and
    \begin{align*}
        \Vert
            V (t)
        \Vert_{X_{1/q, p, q}}
        = O(e^{ - \mu_\ast t}).
    \end{align*}
    for some constant $C>0$.

\end{proof}

\section{A Priori Estimate} \label{sec_a_priori_estimate}
\subsection{Global well-posedness and a priori estimate}

\subsubsection{$L^2$-estimates}
\begin{proposition} \label{prop_trigonal_estimate}
    Let $p \in (1, \infty)$ and $r_1, r_2 \in (1 - 1/p, \infty)$ satisfying $1 - 1/p \geq 1/r_1 + 1/r_2$.
    Then
    \begin{align*}
        &\left|
            \int_\Omega
                \left(
                    \int_{-l}^{x_3}
                        f
                    dz
                \right)
                g h
            dx
        \right| \\
        & \leq C \Vert
            f
        \Vert_{L^2(\Omega)}^{\frac{2}{p}}
        \Vert
            f
        \Vert_{H^1_{x^\prime}L^2_{x_3}(\Omega)}^{1 - \frac{2}{p}}
        \Vert
            g
        \Vert_{L^2(\Omega)}^{\frac{2}{r_2}}
        \Vert
            g
        \Vert_{H^1_{x^\prime}L^2_{x_3}(\Omega)}^{1 - \frac{2}{r_1}}
        \Vert
            h
        \Vert_{L^2(\Omega)}^{\frac{2}{r_2}}
        \Vert
            h
        \Vert_{H^1_{x^\prime}L^2_{x_3}(\Omega)}^{1 - \frac{2}{r_2}}
    \end{align*}
    for some constant $C>0$ and all $f, g, h \in H^1(\Omega)$.
\end{proposition}
\begin{proof}
    The Sobolev inequalities and the interpolation inequalities yield
    \begin{align*}
        & \left|
            \int_\Omega
                \left(
                    \int_{-l}^{x_3}
                        f
                    dz
                \right)
                g h
            dx
        \right| \\
        & \leq \left \Vert
            \int_{-l}^{0}
                |f|
            dz
        \right \Vert_{L^p_{x^\prime}}
        \left \Vert
            \int_{-l}^{0}
                |g| |h|
            dz
        \right \Vert_{L^{p^\prime}_{x^\prime}} \\
        & \leq C \int_{-l}^0
            \Vert
                f
            \Vert_{L^2(\Omega)^2}^{\frac{2}{p}}
            \Vert
                f
            \Vert_{H^1_{x^\prime}(\Torus^2)}^{\frac{2}{p}}
        dz
        \int_{-l}^0
            \Vert
                g
            \Vert_{L^2_{x^\prime}(\Torus^2)}^{\frac{2}{r_2}}
            \Vert
                g
            \Vert_{H^1_{x^\prime}(\Torus^2)}^{1 - \frac{2}{r_1}}
            \Vert
                h
            \Vert_{L^2_{x^\prime}(\Torus^2)}^{\frac{2}{r_2}}
            \Vert
                h
            \Vert_{H^1_{x^\prime}(\Torus^2)}^{1 - \frac{2}{r_2}}
        dz \\
        & \leq C \Vert
            f
        \Vert_{L^2(\Omega)^2}^{\frac{2}{p}}
        \Vert
            f
        \Vert_{H^1_{x^\prime}L^2_{x_3}(\Omega)^2}^{1 - \frac{2}{p}}
        \Vert
            g
        \Vert_{L^2(\Omega)^2}^{\frac{2}{r_2}}
        \Vert
            g
        \Vert_{H^1_{x^\prime}L^2_{x_3}(\Omega)^2}^{1 - \frac{2}{r_1}}
        \Vert
            h
        \Vert_{L^2(\Omega)^2}^{\frac{2}{r_2}}
        \Vert
            h
        \Vert_{H^1_{x^\prime}L^2_{x_3}(\Omega)^2}^{1 - \frac{2}{r_2}}.
    \end{align*}
\end{proof}
We use $L^\infty_t H^1_x$- and $L^\infty_t H^2_x$-estimates and the exponential decay by Hieber and Kashiwabara \cite{HieberKashiwabara2016} and Giga et al \cite{GigaGriesHieberHusseinKashiwabara2017_analiticity}.
We use integration by parts to (\ref{eq_nudging}) and Proposition \ref{prop_trigonal_estimate} to see that
\begin{align*}
    & \frac{\partial_t}{2} \Vert
        V
    \Vert_{L^2 (\Omega)^2}^2
    + \Vert
        \nabla V
    \Vert_{L^2(\Omega)^2}^2 \\
    & \leq \left |
        \int_\Omega
            (U \cdot \nabla v) \cdot V
            + F \cdot V
        dx
    \right |
    - \int_\Omega
        \mu J_\delta V \cdot V
    dx\\
    & \leq \Vert
        V
    \Vert_{L^2(\Omega)^2}
    \Vert
        \nabla_H v
    \Vert_{L^6(\Omega)^4}
    \Vert
        V
    \Vert_{L^3(\Omega)^2} \\
    & + \Vert
        \nabla_H V
    \Vert_{L^2(\Omega)^4}
    \Vert
        \partial_3 v
    \Vert_{L^2(\Omega)^2}^{\frac{1}{2}}
    \Vert
        \partial_3 v
    \Vert_{H^1(\Omega)^2}^{\frac{1}{2}}
    \Vert
        V
    \Vert_{L^2(\Omega)^2}^{\frac{1}{2}}
    \Vert
        V
    \Vert_{H^1(\Omega)^2}^{\frac{1}{2}} \\
    & + \Vert
        F
    \Vert_{L^2(\Omega)^2}
    \Vert
        \nabla V
    \Vert_{L^2(\Omega)^2}
    - \mu \Vert
        V
    \Vert_{L^2(\Omega)^2}^2
    + \mu \delta \Vert
        \nabla V
    \Vert_{L^2(\Omega)^2}^2.
\end{align*}
Using the interpolation inequality and the embedding
\begin{align*}
    \Vert
        \varphi
    \Vert_{L^3(\Omega)^2}
    \leq C \Vert
    \varphi
    \Vert_{L^2(\Omega)^2}^{\frac{1}{2}}
    \Vert
        \varphi
    \Vert_{\dot{H}^1(\Omega)}^{\frac{1}{2}} ,\quad
    \Vert
        \varphi
    \Vert_{L^6(\Omega)}
    \leq C \Vert
        \varphi
    \Vert_{H^1(\Omega)}
\end{align*}
for $\varphi \in H^1(\Omega)$ and taking $\delta > 0$ sufficiently small, we have
\begin{align*}
    & \partial_t \Vert
        V
    \Vert_{L^2 (\Omega)^2}^2
    + \Vert
        \nabla V
    \Vert_{L^2(\Omega)^2}^2 \\
    & \leq C \left(
        \Vert
            \nabla_H v
        \Vert_{H^1(\Omega)}^{4/3}
        + \Vert
            \nabla v
        \Vert_{L^2(\Omega)^6}
        \Vert
            \nabla v
        \Vert_{H^1(\Omega)^6}
        + \Vert
            \nabla v
        \Vert_{L^2(\Omega)^6}^2
        \Vert
            \nabla v
        \Vert_{H^1(\Omega)^6}^2
    \right)
    \Vert
        V
    \Vert_{L^2(\Omega)^2}^2 \\
    & - \mu
    \Vert
        V
    \Vert_{L^2(\Omega)^2}^2
    + C \delta \Vert
        \nabla f
    \Vert_{L^2(\Omega)^2}^2.
\end{align*}
Since $v \in L^\infty (0, \infty; H^2(\Omega)^2)$, we take $\mu > 0$ sufficiently large and apply the Gronwall inequality to obtain
\begin{align*}
    \Vert
        V (t)
    \Vert_{L^2(\Omega)^2}^2
    \leq e^{- ct} \Vert
        V_0
    \Vert_{L^2(\Omega)^2}^2
    + C \delta \int_0^t
        e^{- c (t - s)}
        \Vert
            \nabla f (s)
        \Vert_{L^2(\Omega)^2}^{2}
    ds,
\end{align*}
and
\begin{align*}
    \Vert
        V (t)
    \Vert_{L^2(\Omega)^2}^2
    + \int_0^t
        \Vert
            \nabla V(s)
        \Vert_{L^2(\Omega)^2}^2
    ds
    \leq \Vert
        V_0
    \Vert_{L^2(\Omega)^2}^2
    + C \delta \int_0^t
        \Vert
            \nabla f (s)
        \Vert_{L^2(\Omega)^2}^{2}
    ds,
\end{align*}
for some $c > 0$.


To establish the $L^\infty_t H^1_x$-estimate for $V$, we first get the $L^\infty_t H^1_{x^\prime}$-estimate for $\overline{V}$.
The $L^\infty_t L^4_x$-estimate used to chancel out the boundary trace terms from integration over $(-l ,0)$. The $L^\infty_t L^2_x$-estimate for $\partial_3 \tilde{V} = \partial_3 (V - \overline{V})$ is used to absorb nonlinear terms in the estimate for $\overline{V}$.
\subsubsection{$L^\infty_t H^1_{x^\prime}$-estimates for $\overline{V}$}
We divide (\ref{eq_diff}) into two parts.
The vector $\overline{V}$ satisfies
\begin{align} \label{eq_var_V}
    \begin{split}
        \partial_t \overline{V} - \Delta_H \overline{V} + \nabla_H \Pi
        & = \overline{F}
        - \mu \overline{J_\delta V}
        - \frac{1}{l} \gamma_- \partial_3 V
        - \overline{N} (v, \overline{V}, \tilde{V}),\\
        \mathrm{div}_H \overline{V}
        & = 0,
    \end{split}
\end{align}
where $\tilde{V} = V - \overline{V}$ and
\begin{align*}
    \overline{N} (v, \overline{V}, \tilde{V})
    & = \overline{v} \cdot \nabla_H \overline{V}
    + \overline{V} \cdot \nabla_H \overline{v}
    - \overline{V} \cdot \nabla_H \overline{V}
    + \frac{1}{l} \int_{-l}^0
        \tilde{u} \cdot \nabla \tilde{V} 
        + \tilde{U} \cdot \nabla \tilde{v} 
        - \tilde{U} \cdot \nabla \tilde{V} 
    dz.
\end{align*}
The vector field $\tilde{V}$ satisfies
\begin{equation} \label{eq_tilde_V}
    \begin{aligned}
        \partial_t \tilde{V} - \Delta \tilde{V}
        & = \tilde{F}
        - \mu \widetilde{J_\delta V}
        - \frac{1}{l}\gamma_- \partial_3 V
        - \tilde{N} (v, \overline{V}, \tilde{V})\\
        \mathrm{div} \, \tilde{U}
        & = 0,
    \end{aligned}
\end{equation}
where $\gamma_-$ is the bottom boundary trace operator,
\begin{align*}
    \tilde{N} (v, \overline{V}, \tilde{V})
    & = u \cdot \nabla \tilde{V}
    + \tilde{v} \cdot \nabla_H \overline{V}
    + U \cdot \nabla \tilde{v}
    + \tilde{V} \cdot \nabla_H \overline{v}
    - U \cdot \nabla \tilde{V}
    - \tilde{V} \cdot \nabla_H \overline{V}. \\
    & - \frac{1}{l} \int_{-l}^0
        \tilde{u} \cdot \nabla \tilde{V} 
        + \tilde{U} \cdot \nabla \tilde{v} 
        - \tilde{U} \cdot \nabla \tilde{V} 
    dz,
\end{align*}
and
\begin{gather*}
    \tilde{W}
    = \int^{x_3}_{-l}
        \mathrm{div}_H \tilde{V}
    dz
    = \int^{x_3}_{-l}
        \mathrm{div}_H V
    dz,
    \quad \tilde{U}
    = (\tilde{V}, \tilde{W}).
\end{gather*}
We use the $L^2_tL^2_x$-maximal regularity of 2-D Stokes operator.
By the same way as \cite{HieberHusseingKashiwabara2016}, we find
\begin{align*}
    & \partial_t \Vert
        \nabla_H \overline{V}
    \Vert_{L^2(\Torus^2)^4}^2
    + \mu \Vert
        \nabla_H \overline{V}
    \Vert_{L^2(\Torus^2)^4}^2
    + \Vert
        \Delta_H \overline{V}
    \Vert_{L^2(\Torus^2)^2}^2
    + \Vert
        \nabla_H \Pi
    \Vert_{L^2(\Torus^2)^2}^2\\
    & \leq C \left \Vert
        \overline{F}
        - \frac{1}{l} \gamma_- \partial_3 V
        - \overline{N}(v, \overline{V}, \tilde{V})
        + \mu \overline{K_\delta V}
    \right \Vert_{L^2(\Torus^2)^2}^2\\
    & \leq C \Vert
        F
    \Vert_{L^2(\Omega)^2}^2
    + C \Vert
        \partial_3 V
    \Vert_{L^2(\Omega)^2}^2
    + \frac{1}{100} \Vert
        \nabla \partial_3 V
    \Vert_{L^2(\Omega)^2}^2 \\
    & + \left \Vert
        \overline{N}(v, \overline{V}, \tilde{V})
    \right \Vert_{L^2(\Torus^2)^2}
    + C \mu \delta \Vert
        \nabla V
    \Vert_{L^2(\Omega)^2}^2.
\end{align*}
By the Sobolev embeddings and interpolation inequalities, we have
\begin{align*}
    & \left \Vert
        \overline{v} \cdot \nabla_H \overline{V}
        + \overline{V} \cdot \nabla_H \overline{v}
        - \overline{V} \cdot \nabla_H \overline{V}
    \right \Vert_{L^2(\Torus^2)^2}^2 \\
    & \leq C \Vert
        \overline{v}
    \Vert_{L^2(\Torus^2)^2}
    \Vert
        \overline{v}
    \Vert_{H^1(\Torus^2)^2}
    \Vert
        \nabla \overline{V}
    \Vert_{L^2(\Torus^2)^6}
    \Vert
        \nabla \overline{V}
    \Vert_{H^1(\Torus^2)^6} \\
    & + C \Vert
        \overline{V}
    \Vert_{L^2(\Torus^2)^2}
    \Vert
        \overline{V}
    \Vert_{H^1(\Torus^2)^2}
    \Vert
        \nabla \overline{v}
    \Vert_{L^2(\Torus^2)^6}
    \Vert
        \nabla \overline{v}
    \Vert_{H^1(\Torus^2)^6} \\
    & + C \Vert
        \overline{V}
    \Vert_{L^2(\Torus^2)^2}
    \Vert
        \overline{V}
    \Vert_{H^1(\Torus^2)^2}
    \Vert
        \nabla \overline{V}
    \Vert_{L^2(\Torus^2)^6}
    \Vert
        \nabla \overline{V}
    \Vert_{H^1(\Torus^2)^6},
\end{align*}
and 
\begin{align*}
    & \left \Vert
        \int_{-l}^0
        \tilde{v} \cdot \nabla_H \tilde{V} 
        + \tilde{V} \cdot \nabla_H \tilde{v} 
        - \tilde{V} \cdot \nabla_H \tilde{V} 
        dz
    \right \Vert_{L^2(\Torus^2)^2}^2 \\
    & \leq C \Vert
        v
    \Vert_{H^2(\Omega)}^2
    \Vert
        \nabla_H \tilde{V}
    \Vert_{L^2(\Omega)^2}^2 \\
    & + C \left (
        \int_{-l}^0
            \Vert
                \tilde{V}
            \Vert_{L^2(\Torus^2)^2}^{\frac{1}{2}}
            \Vert
                \tilde{V}
            \Vert_{H^1(\Torus^2)^2}^{\frac{1}{2}}
            \Vert
                \nabla_H \tilde{v}
            \Vert_{L^2(\Torus^2)^4}^{\frac{1}{2}}
            \Vert
                \nabla_H \tilde{v}
            \Vert_{H^1(\Torus^2)^4}^{\frac{1}{2}}
        dz
    \right )^2 \\
    & + C \Vert
        |\tilde{V}| | \nabla \tilde{V}|
    \Vert_{L^2(\Omega)}^2 \\
    & \leq C \Vert
        v
    \Vert_{H^2(\Omega)}^2
    \Vert
        \nabla_H \tilde{V}
    \Vert_{L^2(\Omega)^2}^2 \\
    & + C \Vert
            \tilde{V}
        \Vert_{L^2(\Omega)^2}
        \Vert
            \tilde{V}
        \Vert_{H^1(\Omega)^2}
        \Vert
            \nabla_H \tilde{v}
        \Vert_{L^2(\Omega)^4}
        \Vert
            \nabla_H \tilde{v}
        \Vert_{H^1(\Omega)^4} \\
    & + C \Vert
        |\tilde{V}| | \nabla \tilde{V}|
    \Vert_{L^2(\Omega)}^2.
\end{align*}
Since
\begin{align} \label{eq_w_to_div_H}
    \int_{-l}^0
        \left (
            \int_{-l}^{x_3}
                \mathrm{div}_H \varphi
            dz
        \right )
        \partial_3 \psi
    dx_3
    = - \int_{-l}^0
        ( \mathrm{div}_H \varphi )
        \psi
    dx_3
\end{align}
for all $\varphi, \psi \in H^1(\Omega)^2$ satisfying $\mathrm{div}_H \overline{\varphi} = 0$, we use the same way as above to get
\begin{align*}
    & \left \Vert
        \int_{-l}^0
        w \partial_z \tilde{V} 
        + \tilde{W} \partial_z \tilde{v} 
        - \tilde{W} \partial_z \tilde{V} 
        dz
    \right \Vert_{L^2(\Torus^2)^2}^2 \\
    & \leq C \Vert
        \tilde{V}
    \Vert_{L^2(\Omega)^2}
    \Vert
        \tilde{V}
    \Vert_{H^1(\Omega)^2}
    \Vert
        \nabla_H \tilde{v}
    \Vert_{L^2(\Omega)^4}
    \Vert
        \nabla_H \tilde{v}
    \Vert_{H^1(\Omega)^4} \\
    & + C \Vert
        v
    \Vert_{H^2}^2
    \Vert
        \nabla_H \tilde{V}
    \Vert_{L^2(\Omega)^2}^2
    + C \Vert
        |\tilde{V}| | \nabla \tilde{V}|
    \Vert_{L^2(\Omega)}^2
\end{align*}
Summing up the above bounds and using the Young inequality, we obtain
\begin{align} \label{eq_var_V_H1}
    \begin{split}
        & \partial_t \Vert
            \nabla_H \overline{V} (t)
        \Vert_{L^2(\Torus^2)^4}^2
        + \mu \Vert
            \nabla_H \overline{V} (t)
        \Vert_{L^2(\Torus^2)^4}^2
        + \frac{1}{2} \Vert
            \Delta_H \overline{V} (t)
        \Vert_{L^2(\Torus^2)^2}^2
        + \Vert
            \nabla_H \overline{\Pi} (t)
        \Vert_{L^2(\Torus^2)^2}^2\\
        & \leq \phi_{\overline{H^1}, 1} (t) \Vert
            \nabla_H \overline{V} (t)
        \Vert_{L^2(\Torus^2)^4}^2
        + \frac{1}{100} \Vert
            \partial_3 \nabla V
        \Vert_{L^2(\Torus^2)^2}^2
        + C \Vert
            \tilde{V}| | \nabla \tilde{V}
        \Vert_{L^2(\Omega)^2}^2
        + \phi_{\overline{H^1}, 2} (t),
    \end{split}
\end{align}
for some $L^1 (0, T)$-functions $\phi_{\overline{H^1}, 1}, \phi_{\overline{H^1}, 2}$ depending on $L^\infty_t H^2_x$- and lower order bounds for $v$.
\subsubsection{$L^\infty_t L^4_x$-estimates for $\tilde{V}$}
Multiplying $|\tilde{V}|^2 \tilde{V}$ and integration by parts yields
\begin{align*}
    & \frac{\partial_t}{4} \Vert
        \tilde{V}
    \Vert_{L^4(\Omega)^2}^4
    + \Vert
        \nabla |\tilde{V}|^2
    \Vert_{L^2(\Omega)^2}^2
    + \Vert
        |\tilde{V}| |\nabla \tilde{V}|
    \Vert_{L^2(\Omega)^2}^2 \\
    & \leq \left |
        \int_\Omega
            \left(
                \tilde{F}
                - \frac{1}{l} \gamma_- \partial_3 V
                - \tilde{N} (v, \overline{V}, \tilde{V})
            \right)
            \cdot |\tilde{V}|^2 \tilde{V}
        dx
    \right |
    - \int_\Omega
        \mu \widetilde{J_\delta V}
        \cdot |\tilde{V}|^2 \tilde{V}
    dx \\
    & =: I_1 + I_2 + I_3 + I_4.
\end{align*}
We first find from the Sobolev inequality that
\begin{align*}
    & I_1
    = \left |
        \int_\Omega
            \tilde{F} \cdot |\tilde{V}|^2 \tilde{V}
        dx
    \right | \\
    & \leq C \Vert
        F
    \Vert_{L^2(\Omega)^2}
    \Vert
        |\tilde{V}|^2
    \Vert_{H^1(\Omega)}
    \Vert
        \tilde{V}
    \Vert_{L^4(\Omega)^2}, \\
    &I_4
    = \int_\Omega
        \mu \widetilde{J_\delta V}
        \cdot |\tilde{V}|^2 \tilde{V}
    dx \\
    & \leq - \mu \Vert
        \tilde{V}
    \Vert_{L^4(\Omega)^2}^4
    + C \mu \delta \Vert
        \nabla V
    \Vert_{L^2(\Omega)^4}
    \Vert
        |\tilde{V}|^2
    \Vert_{H^1(\Omega)}
    \Vert
        \tilde{V}
    \Vert_{L^4(\Omega)^2}, 
\end{align*}
By the same way as \cite{HieberKashiwabara2016}, we have
\begin{align*}
    & I_2
    = \left |
        \int_\Omega
            \gamma_- \partial_3 V \cdot |\tilde{V}|^2 \tilde{V}
        dx
    \right |
    \leq \left |
        \int_{\Torus^2}
        \gamma_- \partial_3 V
        \cdot
            \int_{-l}^0
                |\tilde{V}|^2 \tilde{V}
            dz
        dx^\prime
    \right | \\
    & \leq C \Vert
        \partial_3 V
    \Vert_{L^2(\Omega)^2}^{\frac{1}{2}}
    \Vert
        \nabla \partial_3 V
    \Vert_{L^2(\Omega)^4}^{\frac{1}{2}}
    \Vert
        |\tilde{V}|^2
    \Vert_{H^1(\Omega)}
    \Vert
        \tilde{V}
    \Vert_{L^4(\Omega)^2}.
\end{align*}
We estimates $I_3$.
We find from the H\"{o}lder inequality and the trivial inequality
\begin{align*}
    r^{\theta}
    \leq 1 + r^{\theta^\prime},
    \quad \text{for $r \in \Real_{\geq 0}$, $1 \leq \theta \leq \theta^\prime$}, 
\end{align*}
that
\begin{align*}
    & \left |
        \int_\Omega
            (\tilde{v} \cdot \nabla_H\overline{V}) \cdot |\tilde{V}|^2 \tilde{V}
        dx
    \right | \\
    & \leq C \Vert
        \nabla_H \overline{V}
    \Vert_{L^2(\Omega)^4}
    \Vert
        \tilde{V}
    \Vert_{L^4(\Omega)^2}^3
    \Vert
        \tilde{v}
    \Vert_{L^4(\Omega)^2} \\
    & \leq C \Vert
        \nabla_H \overline{V}
    \Vert_{L^2(\Omega)^4}
    (
        1
        + \Vert
            \tilde{V}
        \Vert_{L^4(\Omega)^2}^4
    )
    \Vert
        v
    \Vert_{H^1(\Omega)^2}.
\end{align*}
We see
\begin{align*}
    &\left |
        \int_\Omega
            (U \cdot \nabla \tilde{v}) \cdot |\tilde{V}|^2 \tilde{V}
        dx
    \right | \\
    & \leq \left |
        \int_\Omega
            (\tilde{V} \cdot \nabla_H \tilde{v} + \tilde{W} \partial_3 \tilde{v}) \cdot |\tilde{V}|^2 \tilde{V}
        dx
    \right |
    + \left |
        \int_\Omega
            (\overline{V} \cdot \nabla_H \tilde{v} + \overline{W} \partial_3 \tilde{v}) \cdot |\tilde{V}|^2 \tilde{V}
        dx
    \right | \\
    & =: I_1^\prime + I_2^\prime + I_3^\prime + I_4^\prime.
\end{align*}
By the Sobolev inequalities and the H\"{o}lder inequality, we have
\begin{align*}    
    I_1^\prime
    & \leq C \Vert
        \nabla_H v
    \Vert_{L^2(\Omega)^4}
    \Vert
        |\tilde{V}|^2
    \Vert_{L^2(\Omega)}^{1/2}
    \Vert
        |\tilde{V}|^2
    \Vert_{H^1(\Omega)}^{3/2}, \\
    I_2^\prime
    + I_4^\prime
    & \leq C \Vert
        \nabla_H \tilde{V}
    \Vert_{L^2(\Omega)^4}
    \Vert
        \partial_3 \tilde{v}
    \Vert_{L^2_{x_3} L^\infty_{x^\prime}(\Omega)^2}
    \Vert
        |\tilde{V}|^2 \tilde{V}
    \Vert_{L^2(\Omega)^2} \\
    & \leq C \Vert
        \nabla_H \tilde{V}
    \Vert_{L^2(\Omega)^4}
    \Vert
        \partial_3 \tilde{v}
    \Vert_{H^1(\Omega)^2}^{3/4}
    \Vert
        \partial_3 \tilde{v}
    \Vert_{H^2(\Omega)^2}^{1/4}\\
    & \quad \quad \quad \times \left(
        \Vert
            \tilde{V}
        \Vert_{L^4(\Omega)^2}^3
        + \Vert
            \tilde{V}
        \Vert_{L^4(\Omega)^2}^{\frac{3}{2}}
        \Vert
            \nabla |\tilde{V}|^2
        \Vert_{L^2(\Omega)^3}^{\frac{3}{4}}
    \right) \\
    & \leq C \Vert
        \nabla_H \tilde{V}
    \Vert_{L^2(\Omega)^4}^{\frac{4}{3}}
    \Vert
        \partial_3 \tilde{v}
    \Vert_{H^1(\Omega)^2}
    \Vert
        \tilde{V}
    \Vert_{L^4(\Omega)^2}^4
    + C \Vert
        \partial_3 \tilde{v}
    \Vert_{H^2(\Omega)^2} \\
    & + C \Vert
        \nabla_H \tilde{V}
    \Vert_{L^2(\Omega)^4}^2
    \Vert
        \partial_3 \tilde{v}
    \Vert_{H^1(\Omega)^2}^2
    \Vert
        \tilde{V}
    \Vert_{L^4(\Omega)^2}^4
    + C \Vert
        \nabla_H \tilde{V}
    \Vert_{L^2(\Omega)^4}^2
    + C \Vert
        \partial_3 \tilde{v}
    \Vert_{H^2(\Omega)^2}^2 \\
    & + \frac{1}{16} \Vert
        \nabla |\tilde{V}|^2
    \Vert_{L^2(\Omega)^2}^2, \\
    I_3^\prime
    & \leq C \Vert
        \tilde{V}
    \Vert_{L^2(\Omega)^2}
    \Vert
        \partial_3 \tilde{v}
    \Vert_{L^2_{x_3} L^\infty_{x^\prime}(\Omega)^2}
    \Vert
        |\tilde{V}|^2 \tilde{V}
    \Vert_{L^2(\Omega)^2} \\
    & \leq C \Vert
        \tilde{V}
    \Vert_{L^2(\Omega)^2}
    \Vert
        \partial_3 \tilde{v}
    \Vert_{H^1(\Omega)^2}^{1/2}
    \Vert
        \partial_3 \tilde{v}
    \Vert_{H^2(\Omega)^2}^{1/2} \\
    & \quad \quad \times \left(
        \Vert
            \tilde{V}
        \Vert_{L^4(\Omega)^2}^3
        + \Vert
            \tilde{V}
        \Vert_{L^4(\Omega)^2}^{\frac{3}{2}}
        \Vert
            \nabla |\tilde{V}|^2
        \Vert_{L^2(\Omega)^3}^{\frac{3}{4}}
    \right) \\
    & \leq C \Vert
        \tilde{V}
    \Vert_{L^2(\Omega)^2}^{\frac{4}{3}}
    \Vert
        \partial_3 \tilde{v}
    \Vert_{H^1(\Omega)^2}^{\frac{2}{3}}
    \Vert
        \tilde{V}
    \Vert_{L^4(\Omega)^2}^4\\
    & + C \Vert
        \nabla_H \tilde{V}
    \Vert_{L^2(\Omega)^4}^{\frac{8}{3}}
    \Vert
        \partial_3 \tilde{v}
    \Vert_{H^1(\Omega)^2}^{\frac{4}{3}}
    \Vert
        \tilde{V}
    \Vert_{L^4(\Omega)^2}^4
    + C \Vert
        \partial_3 \tilde{v}
    \Vert_{H^2(\Omega)^2}^2 \\
    & + \frac{1}{16} \Vert
        \nabla |\tilde{V}|^2
    \Vert_{L^2(\Omega)^3}^2.
\end{align*}
The term $\int_\Omega (\tilde{V} \cdot \nabla_H \overline{v}) \cdot |\tilde{V}|^2 \tilde{V} dx$ can be bounded by the same way to bound $I_1^\prime$.
By the same as \cite{HieberKashiwabara2016}, we see that
\begin{align*}
    & \left |
        \int_\Omega
            (\tilde{V} \cdot \nabla_H \overline{V}) \cdot |\tilde{V}|^2 \tilde{V}
        dx
    \right | \\
    & \leq C \Vert
        \nabla\overline{V}
    \Vert_{L^2(\Torus^2)^6}
    \left \Vert
        \int_{-l}^0
            |\tilde{V}|^4
        dz
    \right \Vert_{L^2(\Torus^2)} \\
    & \leq C \Vert
        \nabla\overline{V}
    \Vert_{L^2(\Torus^2)^6}
    \Vert
        |\tilde{V}|^2
    \Vert_{L^2(\Omega)}
    (
        \Vert
            |\tilde{V}|^2
        \Vert_{L^2(\Omega)}
        + \Vert
            \nabla_H |\tilde{V}|^2
        \Vert_{L^2(\Omega)^2}
    ) \\
    & \leq C \Vert
        \nabla \overline{V}
    \Vert_{L^2(\Omega)^6}
    \Vert
        \tilde{V}
    \Vert_{L^4(\Omega)^2}^4
    + C \Vert
        \nabla \overline{V}
    \Vert_{H^1(\Omega)^6}^2
    \Vert
        \tilde{V}
    \Vert_{L^4(\Omega)^2}^4
    + \frac{1}{16} \Vert
        \nabla_H |\tilde{V}|^2
    \Vert_{L^2(\Omega)^2}^2.
\end{align*}
We use (\ref{eq_w_to_div_H}) to get
\begin{align*}
    & \left |
        \int_\Omega
            \int_{-l}^0
                \tilde{u} \cdot \nabla \tilde{V} 
            dz
            \cdot |\tilde{V}|^2 \tilde{V}
        dx
    \right | \\
    & \leq \left |
        \int_\Omega
            \int_{-l}^0
                \tilde{v} \cdot \nabla_H \tilde{V} 
            dz
            \cdot |\tilde{V}|^2 \tilde{V}
        dx
    \right |
    + \left |
        \int_\Omega
            \int_{-l}^0
                (\mathrm{div}_H \tilde{v}) \tilde{V} 
            dz
            \cdot |\tilde{V}|^2 \tilde{V}
        dx
    \right | \\
    & \leq C \Vert
        \tilde{v}
    \Vert_{L^\infty(\Omega)^2}
    \Vert
        \nabla_H \tilde{V}
    \Vert_{L^2(\Omega)^4}
    \Vert
        |\tilde{V}|^2
    \Vert_{L^2(\Omega)}
    \Vert
        \nabla_H |\tilde{V}|^2
    \Vert_{L^2(\Omega)^2}^{\frac{1}{2}} \\
    & + C \Vert
        \tilde{v}
    \Vert_{L^3(\Omega)^2}
    \Vert
        \tilde{V}
    \Vert_{H^1(\Omega)^2}
    \Vert
        |\tilde{V}|^2
    \Vert_{L^2(\Omega)}
    \Vert
        \nabla_H |\tilde{V}|^2
    \Vert_{L^2(\Omega)^2}^{\frac{1}{2}}.
\end{align*}
By the same way as above, we deduce
\begin{align*}
    & \left |
        \int_\Omega
            \int_{-l}^0
                \tilde{U} \cdot \nabla \tilde{v} 
            dz
            \cdot |\tilde{V}|^2 \tilde{V}
        dx
    \right | \\
    & \leq \left |
        \int_\Omega
            \int_{-l}^0
                \tilde{V} \cdot \nabla_H \tilde{v} 
            dz
            \cdot |\tilde{V}|^2 \tilde{V}
        dx
    \right |
    + \left |
        \int_\Omega
            \int_{-l}^0
                (\mathrm{div}_H \tilde{V}) \tilde{v} 
            dz
            \cdot |\tilde{V}|^2 \tilde{V}
        dx
    \right |\\
    & \leq C \Vert
        \tilde{V}
    \Vert_{L^2(\Omega)^2}
    \Vert
        \tilde{v}
    \Vert_{H^3(\Omega)^2}
    \Vert
        |\tilde{V}|^2
    \Vert_{L^2(\Omega)}
    \Vert
        \nabla_H |\tilde{V}|^2
    \Vert_{L^2(\Omega)^2}^{\frac{1}{2}} \\
    & + C \Vert
        \nabla \tilde{V}
    \Vert_{L^2(\Omega)^3}
    \Vert
        \tilde{v}
    \Vert_{H^2(\Omega)^2}
    \Vert
        |\tilde{V}|^2
    \Vert_{L^2(\Omega)}
    \Vert
        \nabla_H |\tilde{V}|^2
    \Vert_{L^2(\Omega)^2}^{\frac{1}{2}},
\end{align*}
and
\begin{align*}
    & \left |
        \int_\Omega
            \int_{-l}^0
                \tilde{U} \cdot \nabla \tilde{V} 
            dz
            \cdot |\tilde{V}|^2 \tilde{V}
        dx
    \right |\\
    & \leq C \int_{-l}^0
        \Vert
            \tilde{V}
        \Vert_{L^4(\Torus^2)^2}
        \Vert
            \nabla_H \tilde{V}
        \Vert_{L^2(\Torus^2)^4}
    dz
    \int_{-l}^0
        \Vert
            |\tilde{V}|^3
        \Vert_{L^4(\Torus^2)}
    dz \\
    & \leq C \Vert
        \tilde{V}
    \Vert_{L^4(\Omega)^2}^2
    \Vert
        \nabla_H \tilde{V}^2
    \Vert_{L^2(\Omega)^4}
    \Vert
        \nabla_H |\tilde{V}|^2
    \Vert_{H^1(\Omega)^2}.
\end{align*}
Summing up the above inequalities and using the Young inequality, we see that there exit some $L^1 (0, T)$-functions $\phi_{\widetilde{L^4}, 1}, \phi_{\widetilde{L^4}, 2}$, and small scaler $\delta_{\widetilde{L^4}} > 0$ such that
\begin{align} \label{eq_tilde_V_L4}
    \begin{split}
        & \frac{\partial_t}{4} \Vert
            \tilde{V} (t)
        \Vert_{L^4(\Omega)^2}^4
        + \mu \Vert
            \tilde{V} (t)
        \Vert_{L^4(\Omega)^2}^4
        + \Vert
            \nabla |\tilde{V} (t)|^2
        \Vert_{L^2(\Omega)^3}^2
        + \Vert
            |\tilde{V}(t)| |\nabla \tilde{V}(t)|
        \Vert_{L^2(\Omega)}^2 \\
        & \leq \phi_{\widetilde{L^4}, 1} (t) \Vert
            \tilde{V} (t)
        \Vert_{L^4(\Omega)^2}^4
        + \delta_{\widetilde{L^4}} \Vert
            \nabla \partial_3 V (t)
        \Vert_{L^2(\Omega)^6}^2
        + \phi_{\widetilde{L^4}, 2} (t)
    \end{split}
\end{align}

\subsubsection{$L^\infty_t L^2_x$-estimates for $\partial_3 V$}
Multiplying $- \partial_3^2 V$ to (\ref{eq_diff}) and integrating over $\Omega$, we have
\begin{align*}
    & \frac{\partial_t}{2} \Vert
         \partial_3 V
    \Vert_{L^2(\Omega)^2}^2
    + \Vert
        \nabla \partial_3 V
    \Vert_{L^2(\Omega)^6}^2\\
    & \leq - \mu \Vert
        \partial_3 V
    \Vert_{L^2(\Omega)^2}^2
    + \left|
        \int_\Omega
            (\partial_3 u \cdot \nabla V) \cdot \partial_3 V
        dx
    \right|
    + \left|
        \int_\Omega
            (\partial_3 U \cdot \nabla v) \cdot \partial_3 V
        dx
    \right| \\
    & + \left|
        \int_\Omega
            (U \cdot \nabla \partial_3 v) \cdot \partial_3 V
        dx
    \right|
    + \left|
        \int_\Omega
            (\partial_3 U \cdot \nabla V) \cdot \partial_3 V
        dx
    \right| \\
    & + \left|
        \int_\Omega
            F \cdot \partial_3^2 V
        dx
    \right|
    + \left|
        \int_\Omega
            K_\delta V \cdot \partial_3^2 V
        dx
    \right| \\
    & =: I_1 + I_2 + I_3 + I_4 + I_5 + I_6 + I_7.
\end{align*}
Using the Sobolev inequality and the interpolation inequality, we have
\begin{align*}
    I_2
    & \leq \Vert
        v
    \Vert_{H^3(\Omega)^2}
    \Vert
        \nabla \tilde{V}
    \Vert_{L^2(\Omega)^6}
    \Vert
        \partial_3 \tilde{V}
    \Vert_{L^2(\Omega)^2},\\
    I_3
    & \leq C \Vert
        v
    \Vert_{H^2(\Omega)^2}
    \Vert
        \partial_3 V
    \Vert_{H^1(\Omega)^2}^{\frac{3}{2}}
    \Vert
        \partial_3 V
    \Vert_{H^1(\Omega)^2}^{\frac{1}{2}} \\
    & + C \Vert
        \nabla_H V
    \Vert_{L^2(\Omega)^2}
    \Vert
        \nabla v
    \Vert_{H^1(\Omega)^2}
    \Vert
        \partial_3 V
    \Vert_{L^2(\Omega)^2}^{\frac{1}{2}}
    \Vert
        \partial_3 V
    \Vert_{H^1(\Omega)^2}^{\frac{1}{2}}.
\end{align*}
We use Proposition \ref{prop_trigonal_estimate} to get
\begin{align*}
    I_4
    & \leq \left|
        \int_\Omega
            (V \cdot \nabla_H \partial_3 v) \cdot \partial_3 V
        dx
    \right|
    + \int_\Omega
        \int_{-l}^0
            |\nabla_H \tilde{V}|
        dz \,
        |\partial_3^2 v| |\partial_3 V|
    dx \\
    & \leq C \Vert
        |V| |\nabla V|
    \Vert_{L^2(\Omega)}
    \Vert
        \nabla_H \partial_3 v
    \Vert_{L^2(\Omega)^6}\\
    &+ C \Vert
        \nabla \tilde{V}
    \Vert_{L^2(\Omega)^6}
    \Vert
        v
    \Vert_{H^2(\Omega)^2}^{\frac{1}{2}}
    \Vert
        v
    \Vert_{H^3(\Omega)^2}^{\frac{1}{2}}
    \Vert
        \partial_3 \tilde{V}
    \Vert_{L^2(\Omega)^2}^{\frac{1}{2}}
    \Vert
        \partial_3 \tilde{V}
    \Vert_{H^1(\Omega)^2}^{\frac{1}{2}}.
\end{align*}
Since
\begin{align*}
    \int_\Omega
        (\partial_3 V \cdot \nabla_H V) \cdot \partial_3 V
    dz
    & = \int_\Omega
        (\partial_3 V \cdot \nabla_H \tilde{V}) \cdot \partial_3 V
    dz
    + \int_\Omega
        (\partial_3 V \cdot \nabla_H \overline{V}) \cdot \partial_3 V
    dz \\
    & = \int_\Omega
        (\mathrm{div}_H \partial_3 V) \tilde{V} \cdot \partial_3 V
    - \int_\Omega
        (\partial_3 V \cdot \tilde{V}) \mathrm{div}_H \partial_3 V
    dz \\
    & + \int_\Omega
        (\partial_3 V \cdot \nabla_H \overline{V}) \cdot \partial_3 V
    dz,
\end{align*}
we use from (\ref{prop_trigonal_estimate}) to get
\begin{align*}
    I_5
    & \leq C \Vert
        |\tilde{V}| |\nabla \tilde{V}|
    \Vert_{L^2(\Omega)}
    \Vert
        \nabla_H \partial_3 \tilde{V}
    \Vert_{L^2(\Omega)^6} \\
    & + C \Vert
        \partial_3 \overline{V}
    \Vert_{L^2(\Omega)^2}
    \Vert
        \partial_3 \tilde{V}
    \Vert_{L^2(\Omega)^2}
    \Vert
        \partial_3 \tilde{V}
    \Vert_{H^1(\Omega)^2}
    + C \Vert
        \partial_3 \tilde{V}
    \Vert_{L^2(\Omega)^2}
    \Vert
        \partial_3 \tilde{V}
    \Vert_{L^2(\Omega)^2}
    \Vert
        \partial_3 \tilde{V}
    \Vert_{H^1(\Omega)^2}.
\end{align*}
The H\"{o}lder inequality yields
\begin{align*}
    I_6 + I_7
    \leq C \Vert
        \nabla F
    \Vert_{L^2(\Omega)^2}
    \Vert
        \partial_3^2 V
    \Vert_{L^2(\Omega)^2}
    + C \Vert
        \nabla V
    \Vert_{L^2(\Omega)^6}
    \Vert
        \partial_3^2 V
    \Vert_{L^2(\Omega)^2}.
\end{align*}
Summing up the above inequalities and using the Young inequality, we see that there exit $L^1(0, T)$-functions $\phi_{\partial_3^{-1}L^2, 1}, \phi_{\partial_3^{-1}L^2, 2}$ and a scalar $\delta_{\partial_3^{-1}L^2} > 0$ such that
\begin{align} \label{eq_L2_del3V}
    \begin{split}
        & \partial_t \Vert
             \partial_3 V (t)
        \Vert_{L^2(\Omega)^2}^2
        + \mu \Vert
             \partial_3 V  (t)
        \Vert_{L^2(\Omega)^2}^2
        + \Vert
            \nabla \partial_3 V  (t)
        \Vert_{L^2(\Omega)^6}^2\\
        & \leq \phi_{\partial_3^{-1}L^2, 1}(t)  \Vert
             \partial_3 V (t)
        \Vert_{L^2(\Omega)^2}^2
        + \phi_{\partial_3^{-1}L^2, 2} (t)
        + \delta_{\partial_3^{-1}L^2} \Vert
            |\tilde{V}| |\nabla \tilde{V}|
        \Vert_{L^2(\Omega)}^2.
    \end{split}
\end{align}
By the estimates (\ref{eq_var_V_H1}), (\ref{eq_tilde_V_L4}), and (\ref{eq_L2_del3V}), we obtain
\begin{align*}
    & \Vert
        \nabla_H \overline{V} (t)
    \Vert_{L^2(\Torus^2)^4}^2
    + \int_0^t
        \Vert
            \Delta_H \overline{V} (s)
        \Vert_{L^2(\Torus^2)^2}^2
        + \Vert
            \nabla_H \overline{\Pi} (t)
        \Vert_{L^2(\Torus^2)^2}^2
    ds \\
    & + \Vert
        \tilde{V} (t)
    \Vert_{L^4(\Omega)^2}^4
    + \int_0^t
        \Vert
            \nabla |\tilde{V} (s)|^2
        \Vert_{L^2(\Omega)^3}^2
        + \Vert
            |\tilde{V}(ts)| |\nabla \tilde{V}(s)|
        \Vert_{L^2(\Omega)}^2
    ds \\
    & + \Vert
         \partial_3 V (t)
    \Vert_{L^2(\Omega)^2}^2
    + \int_0^t
        \Vert
            \nabla \partial_3 V (s)
        \Vert_{L^2(\Omega^6)}^2
    ds \\
    & \leq C e^{-\mu c t + \phi_1 (t)}
    + \phi_2 (t)
\end{align*}
for some $\phi_1, \phi_2 \in L^\infty(0,T)$.

\subsubsection{$L^\infty_t H^1_x$-estimates for $V$}
Multiplying $\Delta V$ and integrating over $\Omega$ and applying the Young inequality, we have
\begin{align*}
    & \frac{\partial_t}{2} \Vert
        \nabla V(t)
    \Vert_{L^2(\Omega)^6}^2
    + \Vert
        \Delta V(t)
    \Vert_{L^2(\Omega)^2}^2 \\
    & \leq - \mu \Vert
        \nabla V(t)
    \Vert_{L^2(\Omega)^6}^2
    + C\Vert
        u \cdot \nabla V
    \Vert_{L^2(\Omega)^2}^2
    + C \Vert
        U \cdot \nabla v
    \Vert_{L^2(\Omega)^2}^2
    + C \Vert
        U \cdot \nabla V
    \Vert_{L^2(\Omega)^2}^2 \\
    & + C \Vert
        \nabla_H \Pi
    \Vert_{L^2(\Omega)^2}^2
    + C \Vert
        F
    \Vert_{L^2(\Omega)^2}^2
    + \frac{1}{4} \Vert
        \Delta V
    \Vert_{L^2(\Omega)^2}^2 \\
    & =: I_1 + I_2 + I_3 + I_4 + I_5 + I_6 + I_7.
\end{align*}
Using the same way as \cite{HieberKashiwabara2016}, we see
\begin{align*}
    I_4
    & \leq C \Vert
        \tilde{V} \cdot \nabla_H \tilde{V}
    \Vert_{L^2(\Omega)^2}^2
    +  C \Vert
        \overline{V} \cdot \nabla_H \tilde{V}
    \Vert_{L^2(\Omega)^2}^2\\
    & +  C \Vert
        \tilde{V} \cdot \nabla_H \overline{V}
    \Vert_{L^2(\Omega)^2}^2
    +  C \Vert
        \tilde{V} \cdot \nabla_H \overline{V}
    \Vert_{L^2(\Omega)^2}^2
    + C \left \Vert
        \int_{-l}^{x_3}
            \mathrm{div}_H \, \tilde{V}
        dz \partial_3 \tilde{V}
    \right \Vert_{L^2(\Omega)^2}^2 \\
    & \leq C \Vert
        |\tilde{V}||\nabla \tilde{V}|
    \Vert_{L^2(\Omega)}^2
    + C \Vert
        \overline{V}
    \Vert_{H^1(\Torus^2)^2}^2
    \Vert
        \nabla_H \tilde{V}
    \Vert_{L^2(\Omega)^4}
    \Vert
        \nabla_H \tilde{V}
    \Vert_{H^1(\Omega)^4} \\
    & + C \Vert
        \tilde{V}
    \Vert_{L^4(\Omega)^2}^2
    \Vert
        \nabla_H \overline{V}
    \Vert_{L^2(\Torus^2)^4}
    \Vert
        \nabla_H \overline{V}
    \Vert_{H^1(\Torus^2)^4} \\
    & + C \Vert
        \overline{V}
    \Vert_{H^1(\Torus^2)^2}^2
    \Vert
        \nabla_H \overline{V}
    \Vert_{L^2(\Torus^2)^4}
    \Vert
        \nabla_H \overline{V}
    \Vert_{H^1(\Torus^2)^4} \\
    & + C \Vert
        \tilde{V}
    \Vert_{L^2(\Omega)^2}^2
    \Vert
        \nabla_H \tilde{V}
    \Vert_{H^1(\Omega)^4}
    \Vert
        \partial_3 \tilde{V}
    \Vert_{L^2(\Omega)^2}
    \Vert
        \partial_3 \tilde{V}
    \Vert_{H^1(\Omega)^2}
\end{align*}
and also
\begin{align*}
    I_2
    + I_3
    & \leq C \Vert
        v
    \Vert_{H^2(\Omega)^2}^2
    \Vert
        \nabla V
    \Vert_{L^2}^2
    + C \Vert
        \nabla_H v
    \Vert_{L^2(\Omega)^4}
    \Vert
        \nabla_H v
    \Vert_{H^1(\Omega)^4}
    \Vert
        \partial_3 V
    \Vert_{L^2(\Omega)^2}
    \Vert
        \partial_3 V
    \Vert_{H^1(\Omega)^2} \\
    & + \Vert
        V
    \Vert_{L^2(\Omega)^2}^2
    \Vert
        v
    \Vert_{H^3(\Omega)^2}^2
    + \Vert
        \nabla_H V
    \Vert_{L^2(\Omega)^4}
    \Vert
        \nabla_H V
    \Vert_{H^1(\Omega)^4}
    \Vert
        \partial_3 v
    \Vert_{L^2(\Omega)^2}
    \Vert
        \partial_3 v
    \Vert_{H^1(\Omega)^2}.
\end{align*}
Therefore, there exit positive functions $\phi_{H^1, 1}\in L^\infty(0, T)$ and $\phi_{H^1, 2} \in L^1(0, T)$ such that
\begin{align}
    & \Vert
        \nabla V(t)
    \Vert_{L^2(\Omega)^2}^2
    + \Vert
        \Delta V(t)
    \Vert_{L^2(\Omega)^2}^2
    \leq \phi_{H^1, 1} (t)
    + \int_0^t \phi_{H^1, 2} (s) ds
\end{align}
\subsubsection{$\mathbb{E}_{1, 1, 2, 2}$-estimate for $V$}
Using the $L^2_t$-$L^2_x$ maximal regularity of $\tilde{A}$, we see
\begin{align*}
    & \Vert
        V
    \Vert_{\mathbb{E}_{1,1, 2, 2}(T)} \\
    & \leq C \Vert
        (\partial_t + \tilde{A})V
    \Vert_{L^2(0, T; L^2(\Omega)^2)} \\
    & \leq C \Vert
        U \cdot \nabla V
    \Vert_{L^2(0, T; L^2(\Omega)^2)}
    + C \Vert
        U \cdot \nabla v
    \Vert_{L^2(0, T; L^2(\Omega)^2)}
    + C \Vert
        u \cdot \nabla V
    \Vert_{L^2(0, T; L^2(\Omega)^2)} \\
    & + C \mu \Vert
        J_\delta V
    \Vert_{L^2(0, T; L^2(\Omega)^2)}
    + C \Vert
        F
    \Vert_{L^2(0, T; L^2(\Omega)^2)}.
\end{align*}
Since
\begin{align*}
    & \Vert
        U \cdot \nabla V
    \Vert_{L^2(0, T; L^2(\Omega)^2)} \\
    & \leq C \Vert
        V
    \Vert_{L^2(0, T; H^2(\Omega)^2)}
    \Vert
        \nabla V
    \Vert_{L^\infty(0, T; H^2(\Omega)^6)}
    + C \Vert
        \nabla V
    \Vert_{L^\infty(0, T; H^2(\Omega)^6)}
    \Vert
        \nabla V
    \Vert_{L^2(0, T; H^1(\Omega)^6)}, \\
    & \Vert
        U \cdot \nabla v
    \Vert_{L^2(0, T; L^2(\Omega)^2)} \\
    &\leq C \Vert
        V
    \Vert_{L^2(0, T; H^2(\Omega)^2)}
    \Vert
        \nabla v
    \Vert_{L^2(0, T; L^2(\Omega)^6)} \\
    & + C \Vert
        \nabla V
    \Vert_{L^2(0, T; L^2(\Omega)^6)}^{\frac{1}{2}}
    \Vert
        \nabla V
    \Vert_{L^\infty(0, T; H^1(\Omega)^6)}^{\frac{1}{2}}
    \Vert
        \nabla v
    \Vert_{L^\infty(0, T; H^2(\Omega)^6)}^{\frac{1}{2}}
    \Vert
        \nabla v
    \Vert_{L^\infty(0, T; H^1(\Omega)^6)}^{\frac{1}{2}},\\
    & \Vert
        u \cdot \nabla V
    \Vert_{L^2(0, T; L^2(\Omega)^2)} \\
    & \leq C \Vert
        v
    \Vert_{L^2(0, T; H^2(\Omega)^2)}
    \Vert
        \nabla V
    \Vert_{L^\infty(0, T; H^2(\Omega)^6)} \\
    & + C \Vert
        \nabla v
    \Vert_{L^\infty(0, T; H^2(\Omega)^6)}^{\frac{1}{2}}
    \Vert
        \nabla v
    \Vert_{L^\infty(0, T; H^1(\Omega)^6)}^{\frac{1}{2}}
    \Vert
        \nabla V
    \Vert_{L^2(0, T; L^2(\Omega)^6)}^{\frac{1}{2}}
    \Vert
        \nabla V
    \Vert_{L^2(0, T; H^1(\Omega)^6)}^{\frac{1}{2}},
\end{align*}
we conclude $V \in \mathbb{E}_{1,1, 2, 2, T}$ for $T > 0$.
Moreover, we find $V \in H^{2,2} (0, T; L^2(\Omega)^2) \cap H^{1,2} (0, T; H^{2, 2}(\Omega)^2$ from $F \in H^{1, 2}(0, T; L^2(\Omega)^2)$ and Proposition \ref{prop_local_wellposedness_maximal_regularity} for $p = q = 2$.

\begin{proof}[Proof of Theorem \ref{thm_main_thoerem}]
    Theorem is the direct consequence of Lemma \ref{lem_pre_of_main_theorem} and a priori bounds in $H^2(\Omega)$.
\end{proof}

\begin{remark}
    The nudging type DA seems not to be used as a practical way for weather or ocean prediction.
    Scientists and engineers in these fields use other types of DA, $e.g.$, the ensemble Kalman filter (EnKF), 3D-VAR, 4D-VAR, etc.
    From the mathematical point of view, one reason may be the regularity of the solution.
    We consider the system like
    \begin{align} \label{eq_u_j+1=Fu_j}
        u(t_j) = F(u(t_j)).
    \end{align}
    In the 3D-VAR scenario, for example, we modify the estimate $\tilde{u}_j$ and error covariance $\tilde{\sigma}_j$ at $t_j>0$ when we obtain the observation $H u_j + (\text{noise})$ at $t_j$, where $H$ is an observation operator and $u_j$ is the true state at $t_j$.
    Then we forward $\tilde{u}_j$ to $t_{j+1}$ by (\ref{eq_u_j+1=Fu_j}), and repeat the procedures of 3D-VAR.
    If we consider the parabolic equation $e.g.$ the Navier-Stokes equations and the primitive equations in a smooth domain, even if $\tilde{u}_j$ is not regular, the linear part $e^{-tA} \tilde{u}_j$ is immediately regularized and then becomes smooth for time and spatial variables.
    On the other hand, in the nudging type DA, we insert the effect of the observation to the right-hand side of the equation as an external force.
    Even if the true solution is smooth, the observation is not always smooth.
    This results the solution to the DA equation gain at most first order time regularity and second order special regularity $i.e.$ $\tilde{u} \in H^{2,p}(0, T;L^q(\Omega)) \cap L^p(0, T;H^{2, q}(\Omega))$ if observation belongs to $L^p(0,T; L^q(\Omega))$.
    The solution $\tilde{u}$ is not smooth enough as some weather and ocean prediction systems based on numerical simulations assume.
    When we seek the solution by numerical simulations, the regularity of the solution is crucial to the accuracy.
    Some ocean model systems require fourth or more order spacial approximation, which means the solution belongs to at least $H^{4, q}(\Omega)$ or $C^4$ regularity for the solution to the DA equation.
    These systems are not compatible with the DA equation based on the nudging equation.
\end{remark}

\begin{appendices}
\section{$L^2_t H^3_x$ estimate for the solution to the primitive equations}
In this appendix, we show the $L^2_t H^3_x$ a priori estimate
\begin{align} \label{eq_L2_H3}
    \Vert
        \Delta v (t)
    \Vert_{L^2(\Omega)^2}
    + \int_0^t
        \Vert
            \nabla \Delta v (t)
        \Vert_{L^2(\Omega)^6}
    ds
    < \infty
\end{align}
for all $t>0$ if $f \in L^2_t H^1_x ((0, \infty) \times \Omega)$.
Although Giga et al \cite{GigaGriesHieberHusseinKashiwabara2017_analiticity} proved $L^\infty_t H^2_x$-estimates and the lower order estimates, they did not show $L^2_t H^3_x$- a priori estimates.
Let $\psi \in H^1_t L^2_{\overline{\sigma}, x}((0, \infty) \times \Omega) \cap  L^2_t H^2_{\sigma, x} ((0, \infty)\times \Omega)$ be solution the hydrostatic Stokes Equations
\begin{align} \label{eq_hydrostatic_Stokes_appendix}
    \begin{split}
        \partial_t \psi + A_2 \psi
        & = g \in L^2_t H^1_{\overline{\sigma}, x} ((0, \infty) \times \Omega), \\
        \psi(0)
        & = \psi_0 \in H^2_{\overline{\sigma}}(\Omega),
    \end{split}
\end{align}
with boundary conditions (\ref{eq_bound_conditions}).
Assume $\psi$ satisfies
\begin{align}
    \int_0^t
        \Vert
            \nabla \partial_t \psi (s)
        \Vert_{L^2 (\Omega)^6}
    ds
    < \infty
\end{align}
for all $t > 0$.
We denote $D^h_j \psi (x) = (\psi(x + h e_j) - \psi(x))/h$ for $j = 1, 2, 3$ and $h > 0$.
Then $D^h_j \psi$ satisfies (\ref{eq_hydrostatic_Stokes_appendix}) with an external force $D^h_j g$ and initial data $D^h_j \psi_0$.
We find
\begin{align*}
    & \int^t_0
        \Vert
            \Delta D^h_j \psi
        \Vert_{L^2(\Omega)^2}^2
    ds
    \leq C \int^t_0
        \Vert
            A_2 D^h_j \psi
        \Vert_{L^2(\Omega)^2}^2
    ds \\
    & \leq C \int^t_0
        \Vert
            D^h_j \partial_t \psi
        \Vert_{L^2(\Omega)^2}^2
    ds
    + C \int^t_0
        \Vert
            D^h_j g
        \Vert_{L^2(\Omega)^2}^2
    ds\\
    & \leq C \int^t_0
        \Vert
            \partial_j \partial_t \psi
        \Vert_{L^2(\Omega)^2}^2
    ds
    + C \int^t_0
        \Vert
            \partial_j g
        \Vert_{L^2(\Omega)^2}^2
    ds.
\end{align*}
Since the right-hand side is uniform for $h$, we take a $\limsup$ to get
\begin{align*}
    \int^t_0
        \Vert
            \partial_j \psi
        \Vert_{H^2(\Omega)^2}^2
    ds
    \leq C \int^t_0
        \Vert
            \partial_j \partial_t \psi
        \Vert_{L^2(\Omega)^2}^2
    ds
    + C \int^t_0
        \Vert
            \partial_j g
        \Vert_{L^2(\Omega)^2}^2
    ds.
\end{align*}
The reader refers to Section 3.1.5 and 4.2.7 of the book \cite{Sohr2001} by Sohr.
Therefore, we see
\begin{align*}
    & \int_0^t
        \Vert
            \partial_j v
        \Vert_{H^2(\Omega)^2}^2
    ds \\
    & \leq  C \int_0^t
        \Vert
            \partial_t \partial_j v
        \Vert_{L^2(\Omega)^2}^2
    ds
    + C \int_0^t
        \Vert
            \partial_j \left(
                v \cdot \nabla_H v + w \partial_3 v
            \right)
        \Vert_{L^2(\Omega)^2}^2
    ds
    + C \int_0^t
        \Vert
            \partial_j f
        \Vert_{L^2(\Omega)^2}^2
    dx.
\end{align*}
The first term is bounded from \cite{GigaGriesHieberHusseinKashiwabara2017_analiticity}.
The Sobolev inequalities and the interpolation inequalities yield
\begin{align*}
    & \int_\Omega
        \left|
            \nabla v \cdot \nabla_H v
        \right|^2
    dx
    \leq C \Vert
        \nabla v
    \Vert_{L^4(\Omega)^6}^2
    \leq C \Vert
        \nabla v
    \Vert_{L^2(\Omega)^6}^{3/2}
    \Vert
        \nabla v
    \Vert_{H^1(\Omega)^6}^{1/2}, \\
    & \int_\Omega
        \left|
            v \cdot \nabla \nabla_H v
        \right|^2
    dx \\
    & \leq C \Vert
        v
    \Vert_{L^\infty(\Omega)^2}
    \Vert
        \nabla^2 v
    \Vert_{L^2(\Omega)^6}
    \leq C \Vert
        v
    \Vert_{H^1(\Omega)^2}^{\frac{1}{2}}
    \Vert
        v
    \Vert_{H^2(\Omega)^2}^\frac{1}{2}
    \Vert
        v
    \Vert_{\dot{H}^2(\Omega)}.
\end{align*}
By the same way, we see
\begin{align*}
    \int_\Omega
        \left |
            \nabla w \partial_3 v
        \right |^2
    dx
    & \leq C \Vert
        \nabla v
    \Vert_{L^2(\Omega)^6}
    \Vert
        \nabla v
    \Vert_{H^1(\Omega)^6}
    \Vert
        \partial_3 v
    \Vert_{L^2(\Omega)^2}
    \Vert
        \partial_3 v
    \Vert_{H^1(\Omega)^2},\\
    & \leq C \Vert
        v
    \Vert_{H^1(\Omega)^2}^2
    \Vert
        \nabla v
    \Vert_{H^1(\Omega)^6}^2\\
    \int_\Omega
        \left |
            w \nabla \partial_3 v
        \right |^2
    dx
    & \leq C \Vert
        \nabla_H v
    \Vert_{L^2(\Omega)^4}
    \Vert
        \nabla_H v
    \Vert_{H^1(\Omega)^4}
    \Vert
        \partial_3 v
    \Vert_{L^2(\Omega)^2}
    \Vert
        \partial_3 v
    \Vert_{H^1(\Omega)^2}\\
    & \leq C \Vert
        v
    \Vert_{H^1(\Omega)^2}^2
    \Vert
        \nabla v
    \Vert_{H^1(\Omega)^6}^2.
\end{align*}
Using the Young inequality, we conclude (\ref{eq_L2_H3}).

\section{Remarks for Semigroup Based Approach}
The main result of this paper is a mathematical justification of the DA in $L^p$-$L^q$ based maximal regularity settings.
However, the analytic semigroup approach such as Hieber and Kashiwabara \cite{HieberKashiwabara2016} can be applicable.
As in \cite{HieberKashiwabara2016}, we set
\begin{align}
    \begin{split}
        & \mathcal{S}_{q, \eta, T}
        = \Set{
            v \in C([0, T]; H^{2/q, q}_{\overline{\sigma}}(\Omega)) \cap C((0, T]; H^{1 + 1/q, q}_{\overline{\sigma}}(\Omega))
        }{\Vert v \Vert_{\mathcal{S}_{q, \eta, T}} < \infty}, \\
        & \Vert v \Vert_{\mathcal{S}_{q, \eta, T}}
        := \sup_{0 < t < T} e^{\eta t} \Vert
            v(t)
        \Vert_{H^{2/q, q}_{\overline{\sigma}}(\Omega)}
        + \sup_{0 < t < T} e^{\eta t} t^{1/2 - 1/2q} \Vert
            v(t)
        \Vert_{H^{1 + 1/q, q}_{\overline{\sigma}}(\Omega)}, \\
        & \widetilde{\mathcal{S}}_{q, \eta, T}
        = \Set{
            v \in C((0, T]; H^{1 + 1/q, q}_{\overline{\sigma}}(\Omega)^2)
        }{\Vert v \Vert_{\widetilde{\mathcal{S}}_{q, \eta, T}} < \infty}, \\
        & \Vert v \Vert_{\widetilde{\mathcal{S}}_{q, \eta, T}}
        := \sup_{0 < t < T} e^{\eta t} t^{1/2 - 1/2q} \Vert
            v(t)
        \Vert_{H^{1 + 1/q, q}_{\overline{\sigma}}(\Omega)}.
    \end{split}
\end{align}
We set $X_q = \Set{\varphi \in H^{2/q, q}_{\overline{\sigma}}(\Omega)^2}{\varphi|_{\Gamma_b = 0}}$.
Using this kind of method based on analytic semigroup theory as \cites{HieberKashiwabara2016,HieberHusseingKashiwabara2016}, we can show
\begin{theorem} \label{thm_supplement_thoerem}
    Let $1 < p < \infty$ and $0 < \alpha < 1$.
    Let $v_0 \in X_q$ be a initial data.
    Let $f \in H^1_{loc} (0, \infty; L^q(\Omega)^2)$ satisfy
    \begin{align*}
        \sup_{t>0} e^{\gamma_0 t} \Vert
            f(s)
        \Vert_{L^q(\Omega)^2}
        + \sup_{0 < s < t} e^{\gamma_0 s} s^{1 - \frac{1}{q}} \Vert
            f (s)
        \Vert_{L^q(\Omega)^2}
        & < \infty,\\
        \Vert
            f
        \Vert_{H^1_{loc}(0, \infty; L^2(\Omega)^2 \cap L^q(\Omega)^2)}
        + \Vert
            e^{\gamma_0 t} f
        \Vert_{L^2(0, \infty; H^1(\Omega)^2)}
        & < \infty,
    \end{align*}
    for some constant $\gamma_0 > 0$.
    Assume $v \in C(0, \infty; X_q)$ is the solution to (\ref{eq_primitive}), which obtained by \cite{HieberHusseingKashiwabara2016} with zero temperature and salinity, satisfying
    \begin{gather} \label{eq_bounds_for_v_2}
        \Vert
            v
        \Vert_{\mathcal{S}_{q, \gamma_1, \infty}}
        < \infty
    \end{gather}
    for some constant $C_0>0$ and $\gamma_1 < \gamma_0$.
    Then there exit $\mu_0, \delta_0>0$, if $\mu \geq \mu_0$ and $\delta \leq \delta_0$, there exits a unique solution $V \in C(0, \infty;X_{1/q, p, q})$ to (\ref{eq_nudging}) such that
    \begin{align*}
        \Vert
            v
        \Vert_{\mathcal{S}_{q, {\mu_\ast}, \infty}}
        < \infty
    \end{align*}
    for some constants $\gamma_0 < \mu_\ast < \gamma_1$ and $C>0$.
    Moreover, $V$ satisfies
    \begin{align*}
        \Vert
            \partial_t V (t)
        \Vert_{L^q(\Omega)^2}
        + \Vert
            V (t)
        \Vert_{D(\tilde{A}_{q, \mu})}
        = O(e^{- \mu_\ast t /2}).
    \end{align*}
\end{theorem}
\begin{remark}
    The assumption on the regularity for the initial data in Theorem \ref{thm_supplement_thoerem} is stronger than Theorem \ref{thm_main_thoerem} since $H^{2/q, q}(\Omega) \hookrightarrow B^{2/q}_{q, p}(\Omega)$ for $p > 2$.
\end{remark}

We show the local well-posedness in $C(0,T_\ast ; L^q(\Omega)^2)$ frame work with the analytic semi-group $e^{- t \tilde{A}}$ and the Fujita-Kato principle.
The mild solution is given by
\begin{align} \label{eq_int_eq_of_V}
    \begin{split}
        V (t)
        &= e^{- t \tilde{A}} V_0
        + \int_0^t
            e^{- (t - s) \tilde{A}} P \left(
                - U (s) \cdot \nabla V (s) + u (s) \cdot \nabla V (s) + U (s) \cdot \nabla v (s)
            \right)
        ds \\
        & + \int_0^t
            e^{- (t - s) \tilde{A}} P F(s)
        ds \\
        & =: N_1(V, v)
        + N_2(F)
        := N(V, v)
    \end{split}
\end{align}

\begin{lemma}\label{lem_local_or_small_data_wellposedness_semigroup_settings}
    Let $T>0$, $0 \leq \tilde{\mu} \leq \mu_\ast$, and $V_0, v_0 \in H^{2/q, q}(\Omega)^2$.
    Let $f \in C(0, T; L^q(\Omega)^2)$ satisfy
    \begin{align*}
        \sup_{0 < s < T} e^{\gamma_0 s} s^{1 - \frac{1}{q}} \Vert
            f (s)
        \Vert_{L^q(\Omega)^2}
        < \infty,
    \end{align*}
    for some constant $\gamma_0 > 0$.
    Let $v$ satisfy the same assumptions as Theorem \ref{thm_supplement_thoerem}. 
    Then, there exists $T_0 > 0$, if $T \leq T_0$, the integral equations (\ref{eq_int_eq_of_V}) admits the unique solution $V \in C(0,T; H^{2/q, q}_{\overline{\sigma}}(\Omega)^2)$ such that
    \begin{align}
        \Vert
            V
        \Vert_{\mathcal{S}_T}
        \leq C.
    \end{align}
    Moreover, there exits $\epsilon_0>0$, if $\Vert v_0 \Vert_{H^{2/q, q}(\Omega)^2}, \sup_{0 < s < T} e^{\gamma_0 s} s^{1 - \frac{1}{q}} \Vert f (s) \Vert_{L^q(\Omega)^2}\leq \varepsilon_0$, it can be taken $T = \infty$.
\end{lemma}
\begin{proof}
    We prove $N : \mathcal{S}_T \rightarrow \mathcal{S}_T$ is a contraction mapping if $T$ or $v_0$ is small.
    We find from Proposition \ref{prop_bilinear_estimate_sobolev} that
    \begin{align} \label{eq_bound_H_2_q}
        \begin{split}
            & e^{\tilde{\mu}t} \Vert
                N_1(V, v)
            \Vert_{H^{2/q, q}(\Omega)^2} \\
            & \leq \int_0^t
                e^{\tilde{\mu}t} e^{- \mu_\ast (t - s)} (t - s)^{ - \frac{1}{q}} \\
            &\times \left(
                    \Vert
                        U (s) \cdot \nabla V (s)
                    \Vert_{L^q(\Omega)^2}
                    + \Vert
                        u (s) \cdot \nabla V (s)
                    \Vert_{L^q(\Omega)^2}
                    + \Vert
                        U (s) \cdot \nabla v (s)
                    \Vert_{L^q(\Omega)^2}
                \right)
            ds \\
            & \leq C \int_0^t
                e^{\tilde{\mu} (t - s)} e^{ - \mu_\ast (t - s)} (t - s)^{ - \frac{1}{q}} s^{- 1 - 1/q}
            ds\\
            & \times \left(
                (\sup_{0 < s < t} e^{\tilde{\mu}s} s^{\frac{1}{2} - \frac{1}{2q}}\Vert
                    V (s)
                \Vert_{H^{1 + 1/q, q}(\Omega)^2})^2
            \right. \\
            & \left.
                + \sup_{0 < s < t} e^{\tilde{\mu}s} s^{\frac{1}{2} - \frac{1}{2q}}\Vert
                    V (s)
                \Vert_{H^{1 + 1/q, q}(\Omega)^2}
                \sup_{0 < s < t} e^{\tilde{\mu}s} s^{\frac{1}{2} - \frac{1}{2q}}\Vert
                    v (s)
                \Vert_{H^{1 + 1/q, q}(\Omega)^2}
            \right)\\
            & \leq C (\sup_{0 < s < t} e^{\tilde{\mu}s} s^{\frac{1}{2} - \frac{1}{2q}}\Vert
                V (s)
            \Vert_{H^{1 + 1/q, q}(\Omega)^2})^2 \\
            & + C \sup_{0 < s < t} e^{\tilde{\mu}s} s^{\frac{1}{2} - \frac{1}{2q}}\Vert
                V (s)
            \Vert_{H^{1 + 1/q, q}(\Omega)^2}
            \sup_{0 < s < t} e^{\tilde{\mu}s} s^{\frac{1}{2} - \frac{1}{2q}}\Vert
                v (s)
            \Vert_{H^{1 + 1/q, q}(\Omega)^2}.
        \end{split}
    \end{align}
    We find from Propositions \ref{prop_pointwise_estimate_semigroup} and \ref{prop_bilinear_estimate_sobolev} that
    \begin{align*}
        & e^{\tilde{\mu}t} \Vert
            N_1(V, v)
        \Vert_{H^{1 + 1/q, q}(\Omega)^2} \\
        & \leq \int_0^t
            e^{\tilde{\mu}t} e^{- \mu_\ast (t - s)} (t - s)^{ - \frac{1}{2} - \frac{1}{2q}} \\
        & \quad \quad \times \left(
                \Vert
                    U (s) \cdot \nabla V (s)
                \Vert_{L^q(\Omega)^2}
                + \Vert
                    u (s) \cdot \nabla V (s)
                \Vert_{L^q(\Omega)^2}
                + \Vert
                    U (s) \cdot \nabla v (s)
                \Vert_{L^q(\Omega)^2}
            \right)
        ds \\
        & \leq C t^{- \frac{1}{2} + \frac{1}{2q}} (\sup_{0 < s < t} e^{\tilde{\mu}s} s^{\frac{1}{2} - \frac{1}{2q}} \Vert
            V (s)
        \Vert_{H^{1 + 1/q, q}(\Omega)^2})^2 \\
        & + C \sup_{0 < s < t} e^{\tilde{\mu}s} s^{\frac{1}{2} - \frac{1}{2q}}\Vert
            V (s)
        \Vert_{H^{1 + 1/q, q}(\Omega)^2}
        \sup_{0 < s < t} e^{\tilde{\mu}s} s^{\frac{1}{2} - \frac{1}{2q}}\Vert
            v (s)
        \Vert_{H^{1 + 1/q, q}(\Omega)^2}.
    \end{align*}
    Similarly we have
    \begin{align*}
        & e^{\tilde{\mu}t} \Vert
            N_1(V_1, v) - N_1(V_2, v)
        \Vert_{H^{1 + 1/q, q}(\Omega)^2} \\
        & \leq C t^{- \frac{1}{2} + \frac{1}{2q}} \sup_{0 < s < t} e^{\tilde{\mu}s} s^{\frac{1}{2} - \frac{1}{2q}}\Vert
            V_1 (s) - V_2 (s)
        \Vert_{H^{1 + 1/q, q}(\Omega)^2} \\
        & \times \left(
            \Vert
                V_1 (s)
            \Vert_{\widetilde{\mathcal{S}}_{q, \tilde{\mu}, T}}
            + \Vert
                V_2 (s)
            \Vert_{\widetilde{\mathcal{S}}_{q, \tilde{\mu}, T}}
            + \Vert
                v (s)
            \Vert_{\widetilde{\mathcal{S}}_{q, \tilde{\mu}, T}}
        \right).
    \end{align*}
    Since $e^{- t \tilde{A}}$ is analytic, we see
    \begin{align*}
        e^{\tilde{\mu}t} \Vert
            e^{- t \tilde{A}} V_0
        \Vert_{H^{2/q, q}(\Omega)^2}
        & \leq C \Vert
            V_0
        \Vert_{H^{2/q, q}(\Omega)^2}, \\
        e^{\tilde{\mu}t} \Vert
            e^{- t \tilde{A}} V_0
        \Vert_{H^{1 + 1/q, q}(\Omega)^2}
        & \leq C t^{- \frac{1}{2} + \frac{1}{2q}} \Vert
            V_0
        \Vert_{H^{2/q, q}(\Omega)^2}.
    \end{align*}
    Moreover, since $D(\tilde{A})$ is densely embedded into $H^{s, q}(\Omega)^2$ for $s \in [0, 2)$, we see
    \begin{align*}
        \Vert e^{- t \tilde{A}} V_0 \Vert_{H^{2/q, q}(\Omega)^2} 
        \rightarrow 0
        \quad \text{as} \quad
        t \rightarrow 0.
    \end{align*}
    We find from Proposition \ref{prop_pointwise_estimate_semigroup} that
    \begin{align*}
        & e^{\tilde{\mu}t} \left \Vert
            \int_0^t
                e^{- (t - s) \tilde{A}} F (s)
            ds
        \right \Vert_{H^{2/q, q}(\Omega)^2} \\
        & \leq C \int_0^t
            e^{\tilde{\mu}t} e^{- \mu_\ast (t - s)} (t - s)^{- \frac{1}{q}} \Vert
                F (s)
            \Vert_{L^q(\Omega)^2}
        ds \\
        & \leq C \sup_{0 < s < t} e^{\tilde{\mu}s} s^{1 - \frac{1}{q}} \Vert
            F (s)
        \Vert_{L^q(\Omega)^2},
    \end{align*}
    and
    \begin{align*}
        & e^{\tilde{\mu}t} \left \Vert
            \int_0^t
                e^{- (t - s) \tilde{A}} F (s)
            ds
        \right \Vert_{H^{1 + 1/q, q}(\Omega)^2} \\
        & \leq \int_0^t
            e^{\tilde{\mu}t} e^{- \mu_\ast (t - s)} (t - s)^{ - \frac{1}{2} - \frac{1}{2q}} \Vert
                F (s)
            \Vert_{L^q(\Omega)^2}
        ds \\
        & \leq C t^{- \frac{1}{2} + \frac{1}{2q}} \sup_{0 < s < t} e^{\tilde{\mu}s} s^{1 - \frac{1}{q}} \Vert
            F (s)
        \Vert_{L^q(\Omega)^2}.
    \end{align*}
    We show continuity of $\int_0^t e^{- (t - s) \tilde{A}} P F (s) ds$ with respect to $t$.
    For $h > 0$
    \begin{align*}
        & \int_0^{t + h}
            e^{- (t + h - s) \tilde{A}} P F (s)
        ds
        - \int_0^t
            e^{- (t - s) \tilde{A}} P F (s)
        ds \\
        & = \int_t^{t + h}
            e^{- (t + h - s) \tilde{A}} P F (s)
        ds
        + \int_0^t
            \left(
                e^{- h \tilde{A}} - \mathrm{id}
            \right) e^{- (t - s) \tilde{A}} P F (s)
        ds \\
        & =: I_1 + I_2.
    \end{align*}
    By the similar estimates as above, we see
    \begin{align*}
        & e^{\tilde{\mu}t} \left \Vert
            I_1
        \right \Vert_{H^{2/q, q}(\Omega)^2}
        \leq C \int_t^{t + h}
            (t - s)^{- \frac{1}{q}} s^{- 1 + \frac{1}{q}}
        ds
        \sup_{0 < s < t} e^{\tilde{\mu}s} s^{1 - \frac{1}{q}} \Vert
            F (s)
        \Vert_{L^q(\Omega)^2}.
    \end{align*}
    Since the integrand is in $L^1(t, t + h)$ uniformly for $h \in (0, h)$, we have $\Vert I_1 \Vert_{H^{2/q, q}(\Omega)^2} \rightarrow 0$ as $h \rightarrow 0$ for $t \geq 0$.
    \begin{align*}
        & e^{\tilde{\mu}t} \left \Vert \left(
                e^{- h \tilde{A}} - \mathrm{id}
            \right) e^{- (t - s) \tilde{A}} P F (s)
        \right \Vert_{H^{2/q, q}(\Omega)^2} \\
        & \leq C (t - s)^{- \frac{1}{q}} s^{- 1 + \frac{1}{q}} \sup_{0 < s < t} e^{\tilde{\mu}s} s^{1 - \frac{1}{q}} \Vert
            F (s)
        \Vert_{L^q(\Omega)^2}
        \in L^1(0,t).
    \end{align*}
    Lebesgue's dominated convergence theorem yields $\left \Vert I_2 \right \Vert_{H^{2/q, q}(\Omega)^2} \rightarrow 0$ for $t \geq 0$.
    Using the similar way, we have $t^{\frac{1}{2} - \frac{1}{2q}}\left \Vert I_1 \right \Vert_{H^{1 + 1/q, q}(\Omega)^2} + t^{\frac{1}{2} - \frac{1}{2q}}\left \Vert I_2 \right \Vert_{H^{1 + 1/q, q}(\Omega)^2} \rightarrow 0$ for $t > 0$.
    The case $h<0$ can be treated similarly.
    Therefore, we see continuity of $\int_0^t e^{- (t - s) \tilde{A}} P F (s) ds$ on $t$.
    We also have continuity $N_1(V, v)$ on $t$ by the similar way.
    Combining the above estimates, we obtain the quadratic inequality
    \begin{align*}
        & \Vert
            N(V, v)
        \Vert_{\widetilde{\mathcal{S}}_{\tilde{\mu}, q, T}} \\
        & \leq C_2 \Vert
            V
        \Vert_{\widetilde{\mathcal{S}}_{\tilde{\mu}, q, T}}^2
        + C_1 \Vert
            V
        \Vert_{\widetilde{\mathcal{S}}_{\tilde{\mu}, q, T}}
        \Vert
            v
        \Vert_{\widetilde{\mathcal{S}}_{\tilde{\mu}, q, T}} \\
        & + C_0 \left(
            \sup_{0 < t < T} e^{\tilde{\mu}t} t^{\frac{1}{2} - \frac{1}{2q}} \Vert
                e^{- t \tilde{A}} V_0
            \Vert_{H^{1 + 1/q, q}(\Omega)^2}
            + C \sup_{0 < s < t} e^{\tilde{\mu}s} s^{1 - \frac{1}{q}} \Vert
                F (s)
            \Vert_{L^q(\Omega)^2}
        \right).
    \end{align*}
    We assume $\Vert v \Vert_{\widetilde{\mathcal{S}}_{\tilde{\mu}, q, T}}
\leq \frac{1}{2C_1}$ is sufficiently small.
    The quadratic estimate implies, if we take $T>0$ so small or $\Vert V_0 \Vert_{H^{2/q, q}(\Omega)^2}$ and $\sup_{0 < s < t} e^{\tilde{\mu}s} s^{1 - \frac{1}{q}} \Vert F (s) \Vert_{L^q(\Omega)^2}$ so small that
    \begin{align*}
        &  R
        := 16 C_2 C_0 \left(
            \sup_{0 < t < T} e^{\tilde{\mu}t} t^{\frac{1}{2} - \frac{1}{2q}} \Vert
                e^{- t \tilde{A}} V_0
            \Vert_{H^{1 + 1/q, q}(\Omega)^2}
            + C \sup_{0 < s < t} e^{\tilde{\mu}s} s^{1 - \frac{1}{q}} \Vert
                F (s)
            \Vert_{L^q(\Omega)^2}
        \right)
        \leq \frac{1}{2},
    \end{align*}
    then
    \begin{align*}
        \Vert
            V
        \Vert_{\widetilde{\mathcal{S}}_{\tilde{\mu}, q, T}}
        \leq \frac{
            1 - \sqrt{
                1 - R
            }
        }{2 C_2}
        =: R_\ast.
    \end{align*}
    We find $N (\cdot, v)$ is an self-mapping on $\Set{V \in \widetilde{\mathcal{S}}_{\tilde{\mu}, q, T}}{\Vert V \Vert_{\widetilde{\mathcal{S}}_{\tilde{\mu}, q, T}} < R_\ast}$ and if we take $R$ and $\Vert v \Vert_{\widetilde{\mathcal{S}}_{\tilde{\mu}, q, T}}$ sufficiently small again, $N (\cdot, v)$ is a contraction mapping.
    Banach's fixed point theorem implies there exits a unique mild solution $V$ to (\ref{eq_int_eq_of_V}) in $\widetilde{\mathcal{S}}_{\tilde{\mu}, q, T}$.
    The estimate (\ref{eq_bound_H_2_q}) implies $N (V, v) \in \mathcal{S}_{\tilde{\mu}, q, T}$, then $V \in \mathcal{S}_{\tilde{\mu}, q, T}$.
\end{proof}
We improve the regularity of the solution $V$ to (\ref{eq_int_eq_of_V}) for regular $F$.

\begin{proposition} \label{prop_exponential_stability_semigroup_with_additional_regularity}
    Let $0 \leq \tilde{\mu} \leq \mu_\ast$, $T > 0$, and $1 < q < \infty$.
    Let $f \in C^\alpha(0, T; L^q(\Omega)^2)$ such that
    \begin{gather}\label{eq_pointwise_estimate_f}
        \begin{split}
            \Vert
                f (t)
            \Vert_{L^q(\Omega)^2}
            \leq C t^{- \beta} e^{-\gamma t}, \\
            \Vert
                f (t + \tau) - f (t)
            \Vert_{L^q(\Omega)^2}
            \leq C \tau^{\alpha} t^{- \beta} e^{-\gamma t},
        \end{split}
    \end{gather}
    for $0 < \alpha < 1$ and $\beta, \gamma > 0$.
    Set $\phi$ such that
    \begin{align*} 
        \phi(t)
        = \int_{0}^t
            e^{- (t - s)\tilde{A}_{\mu, q}} P f
        ds.
    \end{align*}
    Then $\phi$ satisfies
    \begin{align*}
        \Vert
            \phi (t)
        \Vert_{L^q(\Omega)^2}
        \leq C t^{1 - \beta} (
            e^{- \gamma t}
            + e^{- \mu_\ast t/2}
        )
    \end{align*}
    and
    \begin{align*}
        &\Vert
            \partial_t \phi (t)
        \Vert_{L^q(\Omega)^2}
        + \Vert
            \tilde{A}_{\mu, q} \phi (t)
        \Vert_{L^q(\Omega)^2} \\
        &\leq C (
            t^{\alpha - \beta}e^{- \gamma t/2}
            + t^{- \beta} e^{- (\mu_\ast /2 + \gamma) t}
            + t^{- \beta} e^{- \mu_\ast t /2}
        )
    \end{align*}
    for some constant $C > 0$.
    Especially, if $\mu_\ast \leq \gamma$, it follows that
    \begin{align*}
        \Vert
            \partial_t \phi (t)
        \Vert_{L^q(\Omega)^2}
        + \Vert
            \tilde{A}_{\mu, q} \phi (t)
        \Vert_{L^q(\Omega)^2}
        = O(t^{-\beta} e^{\mu_\ast/2}).
    \end{align*}
\end{proposition}
\begin{proof}
    We first observe
    \begin{align*}
        & \left \Vert
            \int_{0}^t
                e^{- (t - s)\tilde{A}_{\mu, q}} P f
            ds
        \right \Vert_{L^q(\Omega)^2}\\
        & \leq \int_{t/2}^t
            s^{- \beta} e^{- \gamma s}
        ds
        + \int_0^{t/2}
            e^{- \mu_\ast t/2} s^{- \beta}
        ds \\
        & \leq C t^{1 - \beta} (
            e^{- \gamma t}
            + e^{- \mu_\ast t/2}
        ).
    \end{align*}
    We split $\tilde{A}_{\mu, q} \phi$ into three parts such that
    \begin{align*}
        \tilde{A}_{\mu, q} \phi
        & = \int_{\frac{t}{2}}^t
            \tilde{A}_{\mu, q} e^{- (t - s) \tilde{A}_{\mu, q}} P \left(
                - f (t) + f(s)
            \right)
        ds \\
        & + \int_{\frac{t}{2}}^t
            \tilde{A}_{\mu, q} e^{- (t - s) \tilde{A}_{\mu, q}} P f (t)
        ds \\
        & + \int_0^{\frac{t}{2}}
            \tilde{A}_{\mu, q} e^{- (t - s) \tilde{A}_{\mu, q}} P f (s)
        ds \\
        & = : I_1 + I_2 + I_3.
    \end{align*}
    By the pointwise estimates (\ref{eq_pointwise_estimate_f}) and the estimate
    \begin{align*}
        & \Vert
            \tilde{A}_{\mu, q} e^{- (t - s) \tilde{A}_{\mu, q}} P f (s)
        \Vert_{L^q(\Omega)^2} \\
        & \leq C e^{- \mu_\ast (t - s)} (t - s)^{-1} \Vert
            f(t) - f(s)
        \Vert_{L^q(\Omega)^2} \\
        & \leq C e^{- \mu_\ast (t - s)} (t - s)^{- 1 +\alpha} s^{- \beta} e^{- \gamma s},
    \end{align*}
    we see
    \begin{align*}
        \Vert
            I_1
        \Vert_{L^q(\Omega)^2}
        \leq C \int_{\frac{t}{2}}^t
            (t - s)^{- 1 + \alpha} s^{- \beta} e^{- \gamma s}
        ds
        \leq C t^{\alpha - \beta} e^{- \gamma t /2}.
    \end{align*}
    By (\ref{eq_pointwise_estimate_f}), we find 
    \begin{align*}
        \Vert
            I_3
        \Vert_{L^q(\Omega)^2}
        \leq C \int_0^{\frac{t}{2}}
            e^{- \mu_\ast t /2} (t - s)^{- 1} s^{- \beta}
        ds
        \leq C t^{- \beta} e^{- \mu_\ast t /2}.
    \end{align*}
    Since
    \begin{align*}
        \int_{\frac{t}{2}}^{t}
            \tilde{A}_{\mu, q} e^{ - (t - s) \tilde{A}_{\mu, q}} P f(t)
        ds
        & = \int_{\frac{t}{2}}^{t}
            \frac{d}{ds} e^{ - (t - s) \tilde{A}_{\mu, q}} P f(t)
        ds \\
        & = P \left(
            e^{- t \tilde{A}_{\tilde{\mu}, q}} f (t) - e^{- \frac{t}{2} \tilde{A}_{\tilde{\mu}, q}} f (t)
        \right)
    \end{align*}
    and $\tilde{A}_{\mu, q}$ is analytic, we have
    \begin{align*}
        \Vert
            I_2
        \Vert_{L^q(\Omega)^2}
        \leq C t^{- \beta} e^{- (\mu_\ast /2 + \gamma) t}.
    \end{align*}
    Since $\Vert \partial_t \phi (t) \Vert_{L^q(\Omega)^2} \leq \Vert \tilde{A}_{\mu, q} \phi (t) \Vert_{L^q(\Omega)^2} + \Vert f (t) \Vert_{L^q(\Omega)^2}$, we have the conclusion.
\end{proof}

\begin{proof}[Proof of Theorem \ref{thm_supplement_thoerem}]
    By Proposition \ref{prop_bilinear_estimate_bilinear_maximal_regularity} for $p = q = 2$ we first deduce that there exists a $L^2$-solution $V \in C(0, \infty; H^1(\Omega)^2) \cap C(0, T; D(\tilde{A}_{2, \mu}))$ such that
    \begin{align*}
        \Vert
            V(t)
        \Vert_{H^1(\Omega)^2}
        \leq C e^{- \mu_\ast t}
    \end{align*}
    for some constant $C>0$.
    Moreover, we repeat the same argument as in the proof of Theorem \ref{thm_main_thoerem} for $p = q = 2$ to see that there exit small $\varepsilon > 0$ and large $T_\ast > 0$ such that $\Vert V (T_\ast) \Vert_{D(\tilde{A}_{2, \mu})} \leq \varepsilon$.
    Applying Lemma \ref{lem_local_or_small_data_wellposedness_semigroup_settings}, we have $V \in \mathcal{S}_{q, {\mu_\ast}, \infty}$.
    By assumption for $f$ and Proposition \ref{prop_bilinear_estimate_sobolev} we have
    \begin{align} \label{eq_decay_of_external_forces}
        \begin{split}
            & \Vert
                F
            \Vert_{L^2(\Omega)}
            = O (e^{- \gamma_0 t}), \\
            & \Vert
                U(t) \cdot \nabla V(t)
            \Vert_{L^2(\Omega)}
            = O(e^{- 2 \mu_\ast t}), \\
            & \Vert
                u(t) \cdot \nabla V(t)
            \Vert_{L^2(\Omega)}
            + \Vert
                U(t) \cdot \nabla v(t)
            \Vert_{L^2(\Omega)}
            = O(e^{- (\gamma_1 + \mu_\ast) t}),
        \end{split}
    \end{align}
    as $t \rightarrow \infty$.
    We use Proposition \ref{prop_exponential_stability_semigroup_with_additional_regularity} to conclude
    \begin{align} \label{eq_decay_of_Lq_solution}
        \Vert
            \partial_t V (t)
        \Vert_{L^2(\Omega)^2}
        + \Vert
            V (t)
        \Vert_{D(\tilde{A}_{2, \mu})}
        = O(e^{- \mu_\ast t /2}).
    \end{align}
    Note that if we take $\mu$ sufficiently large, the decay rate $e^{- \mu_\ast t /2}$ is larger than that of $\Vert \partial_t v (t) \Vert_{L^2(\Omega)^2} + \Vert v(t) \Vert_{D(A_2)}$.
    Proposition \ref{lem_local_or_small_data_wellposedness_semigroup_settings} implies there exists a unique solution in $\mathcal{S}_{q, \mu_\ast, T_\ast}^\prime$ for small $T_\ast^\prime > 0$ and all $1 < q < \infty$.
    We consider the case $q > 2$.
    In this case we have $V(t) \in H^{1 + 1/q,q}(\Omega)^2 \hookrightarrow H^1(\Omega)^2$ for $0 < t < T_\ast^\prime$.
    We can extend $V$ as $H^1$-solution with initial data $V(T_\ast^\prime)$ such that
    \begin{gather} \label{eq_estimate_to_extend_in_L2}
        \begin{split}
            V \in \mathcal{S}_{q, \mu_\ast, T_\ast^\prime} \cap C(T_\ast^\prime/2, \infty; H^1_{\overline{\sigma}}(\Omega)) \cap C(T_\ast^\prime/2, \infty; H^2(\Omega)^2), \\
            \Vert
                V (t)
            \Vert_{H^2(\Omega)^2
            }
            = O(e^{- \mu_\ast t /2})
            \quad \text{as} \rightarrow \infty.
        \end{split}
    \end{gather}
    Since $H^2(\Omega)^2 \hookrightarrow H^{2/q, q}(\Omega)^2$, we see $V \in \mathcal{S}_{q, \mu_\ast, \infty}$.
    By Proposition \ref{prop_bilinear_estimate_sobolev}, we have the same order decay estimate as (\ref{eq_decay_of_external_forces}) even for $L^q$ cases and conclude
    \begin{align*}
        \Vert
            \partial_t V (t)
        \Vert_{L^q(\Omega)^2}
        + \Vert
            V (t)
        \Vert_{D(\tilde{A}_{q, \mu})}
        = O(e^{- \mu_\ast t /2}).
    \end{align*}
    for $q > 2$.
    We consider the case $q < 2$.
    Since $f \in L^2(0, \infty; L^2)$ we find from the same type of bootstrapping argument by \cite{HieberHusseingKashiwabara2016} that $V$ is also $L^2$-solution satisfying (\ref{eq_estimate_to_extend_in_L2}).
    Since $D(\tilde{A}_{\mu_\ast, 2}) \hookrightarrow X_q$, we can extend the solution $V$ globally in $X_q$ and find $V$ satisfies (\ref{eq_decay_of_Lq_solution}).
    We proved the theorem.
\end{proof}
\end{appendices}

\end{document}